\documentclass[11pt, reqno]{amsart}

\usepackage[margin=0.9in]{geometry}

\usepackage[utf8]{inputenc}
\usepackage[T1]{fontenc}

\usepackage{subfiles}
\usepackage{float}
\usepackage{marginnote}
\usepackage{euscript}
\usepackage[dvipsnames]{xcolor}   		             		
\usepackage{graphicx}	
\usepackage{amssymb}
\usepackage{mathrsfs}
\usepackage{amsthm}
\usepackage{amsmath, amsrefs}
\usepackage[foot]{amsaddr}
\usepackage{mathtools}
\usepackage[cal=cm]{mathalfa}
\usepackage{stmaryrd}
\usepackage{upgreek}

\providecommand{\bm}[1]{\boldsymbol{#1}}
\usepackage{bbm}
\usepackage{xypic, dynkin-diagrams}

\usepackage[hypertexnames=false, hyperfootnotes=false, colorlinks=true,linktocpage=true, allcolors=highlight, bookmarksdepth = 2]{hyperref}
\usepackage{cleveref}

\usepackage{url}
\usepackage{float}

\usepackage[full]{textcomp}

\usepackage{subcaption} 
\makeatletter
\newcommand{\citecomment}[2][]{\citen{#2}#1\citevar}
\newcommand{\citeone}[1]{\citecomment{#1}}
\newcommand{\citetwo}[2][]{\citecomment[,~#1]{#2}}
\newcommand{\citevar}{\@ifnextchar\bgroup{;~\citeone}{\@ifnextchar[{;~\citetwo}{]}}}
\newcommand{\citefirst}{\@ifnextchar\bgroup{\citeone}{\@ifnextchar[{\citetwo}{]}}}

\makeatother

\usepackage{tikz,tikz-cd,tikz-3dplot}
\usepackage{pgfplots}
\pgfplotsset{compat=1.18}
\usetikzlibrary{calc}
\usetikzlibrary{fadings}
\usetikzlibrary{decorations.pathmorphing}
\usetikzlibrary{decorations.pathreplacing}
\usetikzlibrary{patterns}
\usetikzlibrary{arrows,shadows,positioning, calc, decorations.markings, 
	hobby,quotes,angles,decorations.pathreplacing,intersections,shapes}
\usetikzlibrary{fillbetween,backgrounds}
\usetikzlibrary{arrows.meta}

\usepackage{xcolor}

\definecolor{highlight}{HTML}{3465a4}

\usepackage{anyfontsize}
\usepackage[lining]{libertine}
\usepackage{courier}

\usepackage{csquotes}
\makeatletter
\renewcommand*\libertine@figurestyle{LF}
\makeatother
\usepackage[libertine,libaltvw,liby]{newtxmath}

\usepackage{microtype}

\usepackage{array}
\newcolumntype{H}{>{\setbox0=\hbox\bgroup}c<{\egroup}@{}}

\BibSpec{article}{+{}{\sc \PrintAuthors} {author}
+{,}{ } {title}
+{,}{ \textit} {journal}
+{,}{ } {volume}
+{}{ \IfEmptyBibField{journal}{}{\PrintDate{date}}} {transition}
+{,}{ } {pages}
+{,}{ } {status}
+{.}{} {transition}
+{}{ \, Preprint available at \tt} {eprint}
+{.}{} {transition}
}

\BibSpec{misc}{+{}{\sc \PrintAuthors} {author}
+{,}{ } {title}
+{,}{ } {date}
+{,}{ } {note}
+{,}{ \url} {url}
+{.}{} {transition}
}

\BibSpec{book}{+{}{\sc \PrintAuthors} {author}
+{,}{ } {title}
+{.}{ \textit} {series}
+{,}{ } {volume}
+{.}{ } {publisher}
+{,}{ } {place}
+{,}{ } {date}
+{.}{} {transition}
}

\BibSpec{incollection}{
+{} {\sc \PrintAuthors} {author}
+{,} { } {title}
+{,} { in \textit} {booktitle}
+{,}{ } {series}
+{,}{ } {volume}
+{}{ \PrintDate } {date}
+{,} { pp.~} {pages}
+{,}{ } {publisher}
+{ }{ Preprint available at \tt} {eprint}
+{.}{} {transition}
}

\BibSpec{inproceedings}{
+{} {\sc \PrintAuthors} {author}
+{,} { } {title}
+{,} { in \textit} {booktitle}
+{,} { pp.~} {pages}
+{,} { \PrintDateB} {date}
+{,} { available at \eprint} {eprint}
+{.} {} {transition}
}

\BibSpec{collection.article}{
+{} {\sc \PrintAuthors} {author}
+{,} { } {title}
+{,} { in \it \PrintConference} {conference}
+{,} { pp.~} {pages}
+{.} {\PrintBook} {book}
+{,} { \PrintDateB} {date}
+{.} {} {transition}
}

\BibSpec{innerbook}{
+{.} { \emph} {title}
+{.} { } {part}
+{:} { \emph} {subtitle}
+{.} { } {series}
+{,} { } {volume}
+{.} { Edited by \PrintNameList} {editor}
+{.} { Translated by \PrintNameList}{translator}
+{.} { \PrintContributions} {contribution}
+{.} { } {publisher}
+{.} { } {organization}
+{,} { } {address}
+{,} { \PrintEdition} {edition}
+{,} { \PrintDateB} {date}
+{.} { } {note}
}

\tikzset{
	commutative diagrams/.cd, 
	arrow style=tikz, 
	diagrams={>=stealth}
}
\tikzset{
	arrow/.pic={\path[tips,every arrow/.try,->,>=#1] (0,0) -- +(0,4pt);},
	pics/arrow/.default={triangle 90}
}
\tikzset{->-/.style={decoration={
			markings,
			mark=at position .6 with {\arrow{latex}}},postaction={decorate}}
}
\tikzset{
	c/.style={every coordinate/.try}
}

\usepackage{enumerate}

\DeclareMathOperator*{\Tot}{Tot}

\usepackage{xifthen}
\newcommand{\Gm}[1][]{
	\ifthenelse{\isempty{#1}}
	{\mathbb{C}^\times}
	{(\mathbb{C}^\times)^{#1}}
}

\newcommand{\ri}{\mathrm{i}}
\newcommand{\re}{\mathrm{e}}

\newcommand{\bbZ}{\mathbb{Z}}

\newcommand{\bbC}{\mathbb{C}}
\newcommand{\bbP}{\mathbb{P}}

\newcommand{\bbQ}{\mathbb{Q}}

\newcommand{\cZ}{\mathcal{Z}}
\newcommand{\cO}{\mathcal{O}}

\newcommand{\cP}{\mathcal{P}}

\newcommand{\DD}{\mathcal{D}}
\newcommand{\LL}{\mathcal{L}}

\newcommand{\cK}{\mathcal{K}}
\newcommand{\cN}{\mathcal{N}}

\newcommand{\cF}{\mathcal{F}}

\newcommand{\RR}{\mathcal{R}}
\newcommand{\cX}{\mathcal{X}}

\newcommand{\MM}{\mathcal{M}}

\newcommand{\cQ}{\mathcal Q}

\DeclareMathOperator{\age}{age}

\DeclareMathOperator{\vir}{vir}

\def\beq{\begin{equation}}                     
	\def\eeq{\end{equation}}                       
\def\bea{\begin{eqnarray}}                     
	\def\eea{\end{eqnarray}}
\def\bary{\begin{array}} 
	\def\eary{\end{array}} 
\def\ben{\begin{enumerate}} 
	\def\een{\end{enumerate}}
\def\bit{\begin{itemize}} 
	\def\eit{\end{itemize}}

\theoremstyle{plain}
\newtheorem{thm}{Theorem}[section]
\newtheorem*{thm*}{Theorem}
\newtheorem{lem}[thm]{Lemma}

\newtheorem{prop}[thm]{Proposition}
\newtheorem*{prop*}{Proposition}
\newtheorem{conj}[thm]{Conjecture}

\newtheorem*{conj*}{Conjecture}

\newtheorem{cor}[thm]{Corollary}
\newtheorem*{cor*}{Corollary}

\theoremstyle{definition}
\newtheorem{defn}[thm]{Definition}
\newtheorem{rmk}[thm]{Remark}

\theoremstyle{plain}

\theoremstyle{plain}

\theoremstyle{plain}

\theoremstyle{definition}

\theoremstyle{plain}

\newenvironment{outlineproof}{\begin{proof}[Outline of proof]}{\end{proof}}

\crefname{equation}{Eq.}{Eqs.}
\crefname{eqnarray}{Eq.}{Eqs.}
\crefname{algo}{algorithm}{algorithms}
\crefname{conj}{conjecture}{conjectures}
\crefname{lem}{lemma}{lemmas}
\crefname{thm}{theorem}{theorems}
\crefname{claim}{claim}{claims}
\crefname{rmk}{remark}{remarks}
\crefname{prop}{proposition}{propositions}
\crefname{section}{section}{sections}
\crefname{appendix}{appendix}{appendices}
\crefname{cor}{corollary}{corollaries}
\crefname{figure}{figure}{figures}
\crefname{table}{table}{tables}
\crefname{example}{example}{examples}
\crefname{prob}{problem}{problems}
\crefname{assm}{assumption}{assumptions}
\crefname{defn}{definition}{definitions}
\crefname{notation}{notation}{notations}
\crefname{speculation}{speculation}{speculations}
\crefname{construction}{construction}{constructions}
\crefname{observation}{observation}{observations}
\crefname{innercustomthm}{Theorem}{Theorems}
\crefname{innercustomconj}{Conjecture}{Conjectures}
\crefname{innercustomassumption}{assumption}{Assumption}
\crefname{innerpracticalresult}{practical result}{practical results}
\newtheorem{property}{Property}

\crefname{equation}{Eq.}{Eqs.}
\crefname{eqnarray}{Eq.}{Eqs.}
\crefname{algo}{Algorithm}{Algorithms}
\crefname{conj}{Conjecture}{Conjectures}
\crefname{lem}{Lemma}{Lemmas}
\crefname{thm}{Theorem}{Theorems}
\crefname{customthm}{Theorem}{Theorems}
\crefname{claim}{Claim}{Claims}
\crefname{rmk}{Remark}{Remarks}
\crefname{prop}{Proposition}{Propositions}
\Crefname{section}{Section}{Sections}
\crefname{appendix}{Appendix}{Appendices}
\crefname{cor}{Corollary}{Corollaries}
\crefname{figure}{Figure}{Figures}
\crefname{table}{Table}{Tables}
\crefname{example}{Example}{Examples}
\crefname{prob}{Problem}{Problems}
\crefname{assm}{Assumption}{Assumptions}
\crefname{defn}{Definition}{Definitions}
\crefname{customconj}{Conjecture}{Conjectures}

\numberwithin{equation}{section}

\title{On orbifold quantum cohomology of foldings of $ADE$ resolutions}
\author{Jingxiang Ma}
\address{University of Sheffield, School of Mathematical and Physical Sciences, Hounsfield Road, Sheffield S3 7RH, United Kingdom}
\email{jingxiangma0114@gmail.com}

\allowdisplaybreaks
\begin{document}

\setcounter{tocdepth}{2}

\begin{abstract}
We compute the untwisted part of the $\mathbb{C}^\times$-equivariant orbifold quantum cohomology of certain finite cyclic quotients of the minimal resolutions of Kleinian singularities, which we refer to as their foldings. We then formulate a conjecture for the full $\mathbb{C}^\times$-equivariant orbifold quantum cohomology, motivated by the Crepant Resolution Conjecture, and provide two pieces of supporting evidence. We also identify the resulting Frobenius structure with known Frobenius structures associated with the corresponding non-simply-laced root systems.
\end{abstract}

\maketitle

\tableofcontents

\section{Introduction} 
Let $\RR$ be a root system and let $\Phi_\RR$ be a finite cyclic group acting on $\RR$ by automorphisms induced by symmetries of its Dynkin diagram. The pairs $(\RR,\Phi_\RR)$ we consider form four classes listed in the first row of \cref{fig:dynkin_symmetries}. For $\RR=\mathrm{A}_{2n-1}$ with $n\geq 2$, the group $\Phi_\RR\cong\bbZ_2$ fixes the marked central vertex and interchanges the two arms. For $\RR=\mathrm{D}_{m+1}$ with $m\geq 3$, the group $\Phi_\RR\cong\bbZ_2$ fixes the horizontal chain, including the marked vertex, and interchanges the two remaining vertices. For $\RR=\mathrm{E}_6$, the group $\Phi_\RR\cong\bbZ_2$ fixes the left two vertices and interchanges the two remaining arms. Finally, for $\RR=\mathrm{D}_4$, the group $\Phi_\RR\cong\bbZ_3$ fixes the marked central vertex and cyclically permutes the three outer vertices.

\newcommand{\DoubleBond}[2]{
  \draw[edge] ($(#1)+(0,0.12)$) -- ($(#2)+(0,0.12)$);
  \draw[edge] ($(#1)+(0,-0.12)$) -- ($(#2)+(0,-0.12)$);
  
  \draw[edge,->] ($(#1)!0.42!(#2)$) -- ($(#1)!0.58!(#2)$);
}
\newcommand{\TripleBond}[2]{
  \draw[edge] ($(#1)+(0,0.16)$) -- ($(#2)+(0,0.16)$);
  \draw[edge] ($(#1)+(0,0.00)$) -- ($(#2)+(0,0.00)$);
  \draw[edge] ($(#1)+(0,-0.16)$) -- ($(#2)+(0,-0.16)$);
  \draw[edge,->] ($(#1)!0.42!(#2)$) -- ($(#1)!0.58!(#2)$);
}

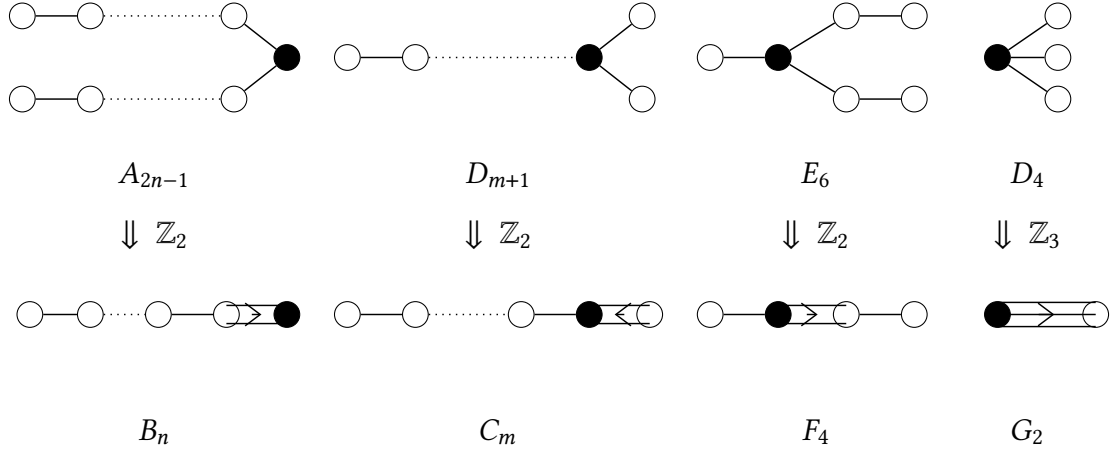
\begin{figure}[H]
\centering
\begin{tikzpicture}[
  x=1cm,y=1cm,
  dynnode/.style={circle,draw,inner sep=0pt,minimum size=3.4mm,line width=0.4pt},
  bnode/.style={dynnode, fill=black}, 
  edge/.style={line width=0.55pt},
  dots/.style={line width=0.55pt,dotted},
  lab/.style={font=\large},
  >={Straight Barb[length=2mm,width=3mm]} 
]

\begin{scope}[xshift=0cm,yshift=0cm]
  \node[bnode] (Ac) at (4.2,0.0) {};

  \node[dynnode] (Au3) at (3.5,0.55) {};
  \node[dynnode] (Au2) at (1.6,0.55) {};
  \node[dynnode] (Au1) at (0.7,0.55) {};
  \draw[edge] (Au1)--(Au2);
  \draw[dots] (Au2)--(Au3);
  \draw[edge] (Au3)--(Ac);

  \node[dynnode] (Al3) at (3.5,-0.55) {};
  \node[dynnode] (Al2) at (1.6,-0.55) {};
  \node[dynnode] (Al1) at (0.7,-0.55) {};
  \draw[edge] (Al1)--(Al2);
  \draw[dots] (Al2)--(Al3);
  \draw[edge] (Al3)--(Ac);

  \node[lab] at (2.45,-1.55) {$A_{2n-1}$};
\end{scope}

\begin{scope}[xshift=4.2cm,yshift=0cm]
  \node[bnode] (Dc) at (4.0,0.0) {};

  \node[dynnode] (D2) at (1.7,0.0) {};
  \node[dynnode] (D1) at (0.8,0.0) {};
  \draw[edge] (D1)--(D2);
  \draw[dots] (D2)--(Dc);

  \node[dynnode] (Du) at (4.7,0.55) {};
  \node[dynnode] (Dd) at (4.7,-0.55) {};
  \draw[edge] (Dc)--(Du);
  \draw[edge] (Dc)--(Dd);

  \node[lab] at (2.8,-1.55) {$D_{m+1}$};
\end{scope}

\begin{scope}[xshift=9.0cm,yshift=0cm]
  \node[dynnode] (E0) at (0.8,0.0) {};
  \node[bnode]   (E1) at (1.7,0.0) {};
  \draw[edge] (E0)--(E1);

  \node[dynnode] (Eu1) at (2.6,0.55) {};
  \node[dynnode] (Eu2) at (3.5,0.55) {};
  \draw[edge] (E1)--(Eu1)--(Eu2);

  \node[dynnode] (Ed1) at (2.6,-0.55) {};
  \node[dynnode] (Ed2) at (3.5,-0.55) {};
  \draw[edge] (E1)--(Ed1)--(Ed2);

  \node[lab] at (2.2,-1.55) {$E_{6}$};
\end{scope}

\begin{scope}[xshift=12.6cm,yshift=0cm]
  \node[bnode] (T0) at (1.0,0.0) {};

  \node[dynnode] (T1) at (1.8,0.55) {};
  \node[dynnode] (T2) at (1.8,0.00) {};
  \node[dynnode] (T3) at (1.8,-0.55) {};
  \draw[edge] (T0)--(T1);
  \draw[edge] (T0)--(T2);
  \draw[edge] (T0)--(T3);

  \node[lab] at (1.4,-1.55) {$D_{4}$};
\end{scope}

\node[lab] at (2.45,-2.35) {$\Downarrow\;\mathbb Z_2$};
\node[lab] at (7.0,-2.35)  {$\Downarrow\;\mathbb Z_2$};
\node[lab] at (11.2,-2.35) {$\Downarrow\;\mathbb Z_2$};
\node[lab] at (14.0,-2.35) {$\Downarrow\;\mathbb Z_3$};

\begin{scope}[xshift=0cm,yshift=-3.4cm]
  \node[dynnode] (B1) at (0.8,0.0) {};
  \node[dynnode] (B2) at (1.6,0.0) {};
  \node[dynnode] (B3) at (2.5,0.0) {};
  \node[dynnode] (B4) at (3.4,0.0) {};  
  \node[bnode]   (B5) at (4.2,0.0) {};  

  \draw[edge] (B1)--(B2);
  \draw[dots] (B2)--(B3);
  \draw[edge] (B3)--(B4);

  \DoubleBond{B4}{B5}

  \node[lab] at (2.45,-1.55) {$B_{n}$};
\end{scope}

\begin{scope}[xshift=4.2cm,yshift=-3.4cm]
  
  \node[dynnode] (C1) at (0.8,0.0) {};
  \node[dynnode] (C2) at (1.7,0.0) {};
  \node[dynnode] (C3) at (3.1,0.0) {};
  \node[bnode]   (C4) at (4.0,0.0) {};  
  \node[dynnode] (C5) at (4.8,0.0) {};  

  \draw[edge] (C1)--(C2);
  \draw[dots] (C2)--(C3);
  \draw[edge] (C3)--(C4);

  \DoubleBond{C5}{C4}

  \node[lab] at (2.8,-1.55) {$C_{m}$};
\end{scope}

\begin{scope}[xshift=9.0cm,yshift=-3.4cm]
  \node[dynnode] (F1) at (0.8,0.0) {};
  \node[bnode]   (F2) at (1.7,0.0) {};
  \node[dynnode] (F3) at (2.6,0.0) {};
  \node[dynnode] (F4) at (3.5,0.0) {};

  \draw[edge] (F1)--(F2);

  \DoubleBond{F2}{F3}

  \draw[edge] (F3)--(F4);

  \node[lab] at (2.2,-1.55) {$F_{4}$};
\end{scope}

\begin{scope}[xshift=12.6cm,yshift=-3.4cm]
  \node[bnode]   (G1) at (1.0,0.0) {};   
  \node[dynnode] (G2) at (2.3,0.0) {};

  \TripleBond{G1}{G2}

  \node[lab] at (1.4,-1.55) {$G_{2}$};
\end{scope}

\end{tikzpicture}
\caption{Dynkin symmetries and foldings}
\label{fig:dynkin_symmetries}
\end{figure}

Associated with such a pair $(\RR, \Phi_{\RR})$ are two non-simply-laced root systems, one is obtained by \textbf{restricting} to the $\Phi_{\RR}$-invariant roots, denoted by $\RR^{\mathrm{fold}}$ and its Dynkin diagram is shown in the second row of \cref{fig:dynkin_symmetries}, another is obtained by \textbf{averaging} roots, denoted by $\RR^{\mathrm{ave}}$. These root systems can be summarised in the following table:

\begin{table}[ht]
\centering
\begin{tabular}{|c|c|c|c|c|}
\hline
$\RR$ & $\Phi_{\RR}$ & $\RR^{\mathrm{fold}}$ & $\RR^{\mathrm{ave}}$ & $N_{\RR^{\mathrm{fold}}}$ \\
\hline
$\mathrm{A}_{2n-1}$ & $\mathbb{Z}_2$ & $\mathrm{B}_n$ & $\mathrm{C}_n$ & 2\\
$\mathrm{D}_{m+1}$  & $\mathbb{Z}_2$ & $\mathrm{C}_m$ & $\mathrm{B}_m$ & $m$\\
$\mathrm{E}_6$      & $\mathbb{Z}_2$ & $\mathrm{F}_4$ & $\mathrm{F}_4$ & 3\\
$\mathrm{D}_4$      & $\mathbb{Z}_3$ & $\mathrm{G}_2$ & $\mathrm{G}_2$ & 2\\
\hline
\end{tabular}
\caption{restricting and averaging root systems under a Dynkin symmetry}
\label{tab:folding_root_systems}
\end{table}

Here, \textbf{averaging} of $\beta\in \mathfrak{h}_{\RR}$, denoted by $\overline{\beta}$, is defined by
\beq
\label{def: averaging}
\overline{\beta} = \frac{1}{|\Phi_{\RR}|}\sum_{g\in \Phi_{\RR}}g\circ \beta.
\eeq

Let $G_\RR < \mathrm{SL}(2,\bbC)$ be the corresponding Kleinian subgroup in the $\mathrm{ADE}$ classification, $\bbC^2/G_\RR$ be the corresponding Kleinian singularity and $Z_\RR := \widetilde{\bbC^2/G_\RR}$ be its minimal resolution. The geometry of $Z_\RR$ is closely related to root system $\RR$ via the McKay correspondence. For example, there is a canonical identification $H_2(Z_\RR;\mathbb{Z})\cong \Lambda_{\RR}$, under which the classes of exceptional curves correspond to simple roots, and the intersection form agrees with the negative of the inner product on $\Lambda_{\RR}$. 

The $\Phi_{\RR}$-action on $\RR$ also admits a geometric interpretation. In \cite{MR584445}, Slodowy shows that there exists a natural $\Phi_{\RR}$-action on $Z_\RR$, such that the exceptional curves are stable under $\Phi_{\RR}$ and the action on the curves is compatible with the action on Dynkin diagrams. Moreover, $\Phi_{\RR}$ is free on the complement of fixed points, giving the same number of isolated $\mathrm{A}_{|\Phi_\RR|-1}$ singularities on the quotient space $Z_\RR/\Phi_{\RR}$. The number of fixed points on $Z_\RR$ is denotes by $N_{\RR^{\mathrm{fold}}}$ and listed in \cref{tab:folding_root_systems}.

Let $\cX_{\RR^{\mathrm{fold}}}= [Z_\RR/\Phi_{\RR}]$ be the corresponding orbifold/quotient stack, which we call the \textbf{folding} of $Z_\RR$. The diagonal $\bbC^{\times}$-action on $\bbC^2$ induces a $\bbC^{\times}$-action on $Z_\RR$, and fixes an irreducible component of the exceptional curves corresponding to the marked node in \cref{fig:dynkin_symmetries}. Moreover, the $\bbC^{\times}$-action descends to $\cX_{\RR^{\mathrm{fold}}}$. Associated with these data is the \textbf{$\bbC^{\times}$-equivariant orbifold quantum cohomology} $\mathrm{QH}_{\bbC^{\times}}^*(\cX_{\RR^{\mathrm{fold}}})$. Relevant theories are reviewed in \cref{sec: CRcohomology}.

\subsection{Main results}

We first study quantum multiplication on the untwisted sector. To state our result, we describe the relevant cohomology and degree lattices. Consider the quotient map $\pi\colon Z_\RR\to \cX_{\RR^{\mathrm{fold}}}$. The pullback
\[
\pi^*\colon H^2(\cX_{\RR^{\mathrm{fold}}};\bbC)\longrightarrow H^2(Z_\RR;\bbC)
\]
is naturally identified with the embedding of root systems $\RR^{\mathrm{fold}}\subset \RR$. Under this identification, the orbifold Poincar\'e pairing agrees with the negative of the inner product, scaled by $\frac{1}{|\Phi_\RR|}$:
\begin{equation}
    H^2(\cX_{\RR^{\mathrm{fold}}};\bbC)\cong \bigg(\mathfrak{h}_{\RR^{\mathrm{fold}}},-\frac{1}{|\Phi_\RR|}\big\langle,\big\rangle\bigg).
    \label{eq: second_cohomology}
\end{equation}

The degree lattice, or equivalently the second homology, of $\cX_{\RR^{\mathrm{fold}}}$ is obtained by averaging curve classes on $Z_\RR$. Thus,

\begin{equation}
    H_2(\cX_{\RR^{\mathrm{fold}}};\bbZ)\cong \Lambda_{\RR^{\mathrm{ave}}}.
    \label{eq: second_homology}
\end{equation}

The root systems $\RR^{\mathrm{fold}}$ and $\RR^{\mathrm{ave}}$ are dual to each other, consistently with the Kronecker pairing on $\cX_{\RR^{\mathrm{fold}}}$.

Under these identifications, and writing $\RR^{\mathrm{ave},+}$ for the set of positive roots of $\RR^\mathrm{ave}$, our first main result is the following.

\begin{thm}
    For $ \phi_i,\phi_j,\tau\in H^2(\cX_{\RR^{\mathrm{fold}}})\cong \mathfrak{h}_{\RR^{\mathrm{fold}}}$, quantum multiplication is given by
    \begin{equation}
        \phi_i\star_{\tau}\phi_j = -\nu^2|G_\RR|\langle \phi_i,\phi_j \rangle + \nu\sum_{\overline{\beta}\in \RR^\mathrm{ave,+}}\langle\overline{\beta},\phi_i\rangle \langle\overline{\beta},\phi_j\rangle  \frac{1+ \re^{-\langle\overline{\beta},\tau\rangle}}{1- \re^{-\langle\overline{\beta},\tau\rangle}}\overline{\beta}^{\vee} \,.
    \end{equation}
 \label{thm: qproduct_main_thm}
\end{thm}

This theorem gives a closed formula for quantum multiplication on the untwisted sector.

We next turn to the full orbifold quantum cohomology, including products involving twisted sectors. To investigate these products, we use the fact that the coarse moduli space of $\cX_{\RR^{\mathrm{fold}}}$ admits a crepant resolution. In this setup, Ruan's \emph{Crepant Resolution Conjecture} \cite{MR2234886} (CRC) suggests comparing this orbifold quantum cohomology with the quantum cohomology of the crepant resolution.

In Bryan-Graber's later formalism of CRC, an explicit affine change of variables is an essential part of the CRC, see \cite{MR2483931}. Although substantial work has been carried out for Kleinian singularities \cite{MR2483931,MR2411404,MR2510741,MR3096798}, such a change does not appear to have been formulated previously for the quotient stacks considered here. This leads us to propose the following conjecture:
\begin{conj}[= \cref{conj:crc-YRp}]
After analytic continuation and an affine change of variables as in \cref{def:affine trans}, there is an isomorphism
\beq
\mathrm{QH}_{\bbC^\times}^*\bigl( Z_{\RR^{\rm res}}\bigr)\ \cong\ \mathrm{QH}^*_{\mathrm{orb},\bbC^\times}\bigl(\cX_{\RR^{\mathrm{fold}}}\bigr).
\eeq
\end{conj}
Here $Z_{\RR^{\rm res}}$ denotes the crepant resolution of the coarse moduli space of $\cX_{\RR^{\mathrm{fold}}}$. This resolution is described in \cref{subsec: crepant_resolution_of_foldings}.

We establish two results that provide evidence for this conjecture. We state both in simplified form; their precise formulations are given in \cref{prop:perroni-evidence-YRp,prop:FM-crc-compatibility}, respectively.
\begin{prop}[=\cref{prop:perroni-evidence-YRp}]
After an explicit specialisation of the exceptional quantum parameters, the quantum-corrected cohomology of $Z_{\RR^{\rm res}}$ is isomorphic, as a ring, to the Chen--Ruan cohomology of $\cX_{\RR^{\mathrm{fold}}}$.
\end{prop}

\begin{prop}[=\cref{prop:FM-crc-compatibility}]
The affine transformation in \cref{conj:crc-YRp} is compatible, via Iritani's integral structure, with the Fourier--Mukai transform, in the sense that it identifies the corresponding integral central charges, up to the specified signs in the $(\mathrm{D}_4,\mathrm{G}_2)$ case.
\end{prop}

The first proposition verifies the \emph{cohomological crepant resolution conjecture} of Ruan in \cite{MR2234886}, which is a cohomological limit predicted by \cref{conj:crc-YRp}. The second proposition provides categorical evidence for the proposed affine transformation under Iritani's formalism of \emph{crepant resolution conjecture with an integral structure} developed in \cite{MR2553377, MR2683208}. Neither proposition proves the full identification of quantum products after analytic continuation though; establishing such an identification is left for future work.

Finally, we relate the untwisted quantum product to known Frobenius structures associated with non-simply-laced root systems. The precise formulations are given in \cref{cor:BG-folding,thm:BCFG-dubrovin-dual}. Let
\[
\iota\colon H^*(\cX_{\RR^{\mathrm{fold}}})\longrightarrow
H^*_{\rm CR}(\cX_{\RR^{\mathrm{fold}}})
\]
denote the inclusion of the untwisted sector. As a direct consequence of \cref{thm: qproduct_main_thm}, we obtain:

\begin{cor}[=\cref{cor:BG-folding}]
After the change of variables
\[
Q^{\bar{\beta}}=\re^{-\langle\overline{\beta},\tau\rangle},
\]
the untwisted quantum cohomology
$\iota^*\!\bigl(\mathrm{QH}_{\bbC^\times}^*(\cX_{\RR^{\mathrm{fold}}})\bigr)$
is isomorphic, as a Frobenius algebra, to the one constructed by Bryan--Gholampour \cite[Theorem~6]{MR2411404} for $\RR^{\mathrm{ave}}$.
\end{cor}

Using in addition the Landau--Ginzburg description of the extended affine Weyl Frobenius manifolds, we prove:
\begin{thm}[=\cref{thm:BCFG-dubrovin-dual}]
The untwisted quantum cohomology
$\iota^*\!\bigl(\mathrm{QH}_{\bbC^\times}^*(\cX_{\RR^{\mathrm{fold}}})\bigr)$
is isomorphic to the Dubrovin dual $\MM^\flat_{\RR^{\mathrm{ave}}}$ of the extended affine Weyl Frobenius manifold $\MM_{\RR^{\mathrm{ave}}}$.
\end{thm}

Thus, the folding construction provides a geometric realisation of these Frobenius structures associated with non-simply-laced root systems.

\subsection{Organisation}
The paper is organised as follows. 

In \cref{sec: CRcohomology}, we recall the necessary background on Chen--Ruan cohomology, equivariant Gromov--Witten theory, and orbifold quantum cohomology for global quotient stacks. 

In \cref{sec:proof_main_theorem}, we prove \cref{thm: qproduct_main_thm}. The proof strategy is outlined in \cref{subsec: sketch_of_proof} and reduces the theorem to two key ingredients: the construction of two threefolds and a multiple-cover calculation. These are established in \cref{subsec: threefolds_with_properties,subsec: multi_cover_formula}, respectively.

In \cref{sec: CRC}, we describe the relevant crepant resolutions and the affine identification of their cohomology groups before formulating \cref{conj:crc-YRp}. The supporting evidence is developed in \cref{subsec: quantum_corrected_ring,subsec: FM_transform}.

Finally, in \cref{sec: BCFG}, we compare \cref{thm: qproduct_main_thm} with the known Frobenius structures associated with non-simply-laced root systems.

\subsection*{Acknowledgements} The author would like to thank the organizers of the Spring School \emph{Group Actions and Symplectic Singularities}, held in Lille in July 2025, where he learned about Slodowy’s work on simple singularities. He also thanks Andrea Brini and Tom Bridgeland for helpful discussions. This work was supported by a PhD studentship from the EPSRC Doctoral Training Partnership EP/W524360/1.

\section{Setup and Preliminaries}
\label{sec: CRcohomology}
In this section, we first the setup and recall the necessary background on Chen--Ruan cohomology, Gromov--Witten theory, quantum cohomology and their equivariant variants.

\subsection{Setup and conventions}
\label{subsec:setup-notation}

All varieties and stacks are defined over the base field $\bbC$. Let $Y$ be a smooth quasi-projective variety and let $G$ be a finite group acting on $Y$. We write
\beq
  \cX=[Y/G],
\eeq
for the associated global quotient stack, and denote by $\pi\colon \cX \to Y/G$ the coarse moduli map.

\smallskip
\noindent\textbf{$\bbC^\times$-action.}
Assume that $\bbC^\times$ acts algebraically on $Y$ and commutes with the $G$-action, so that it induces an action on $\cX$.
The induced $\bbC^\times$-action on the inertia stack $I\cX$ is described in \eqref{def: inertia_stack}.
We assume moreover that the fixed locus $Y^{\bbC^\times}$ is compact.
This hypothesis ensures that the localised $\bbC^\times$-equivariant pushforwards in equivariant Chen--Ruan cohomology are well-defined
(see \cref{subsec:equivariant-variants}).

\smallskip
\noindent\textbf{Coefficients of (co)homology.}
Unless stated otherwise, all (co)homology groups are taken with $\bbQ$-coefficients.
When a $\bbC^\times$-action is present, we work $\bbC^\times$-equivariantly and denote equivariant cohomology by
$H^*_{\bbC^\times}(\,\cdot\,)$.
In this setting, all classes are understood as equivariant lifts; moreover, whenever we apply localisation we implicitly extend scalars from
$H^*_{\bbC^\times}(\mathrm{pt})\cong \bbQ[\nu]$ to its field of fractions $\bbQ(\nu)$.

\smallskip
\noindent\textbf{(Co)homology of quotient stack and the Kronecker pairing.}
Let $q\colon Y\to Y/G$ be the quotient map.
With $\bbQ$-coefficients we identify $H^*(Y/G)$ with the $G$-invariant subspace $H^*(Y)^G$ via pullback $q^*$.
On homology we use the averaging projector
\beq
  {\rm Av}\colon H_2(Y;\bbQ)\to H_2(Y;\bbQ)^G,\qquad
  {\rm Av}(\gamma)\coloneqq \frac{1}{|G|}\sum_{g\in G} g_*(\gamma),
\eeq
and we regard the degree lattice for the quotient as ${\rm Av}\big(H_2(Y;\bbZ)\big)\subset H_2(Y;\bbQ)^G$.
The Kronecker pairing on the quotient is defined by evaluation upstairs:
for $\alpha\in H^k(Y/G)$ and $\overline{\beta}\in {\rm Av}\big(H_2(Y;\bbZ)\big)$,
\beq
  \langle \alpha,\overline{\beta}\rangle_{Y/G}\ \coloneqq\ \langle q^*\alpha,\overline{\beta}\rangle_Y.
\eeq

\subsection{Chen--Ruan cohomology of orbifolds}
\label{subsec:CR}

A \emph{complex orbifold} generalises the notion of a complex manifold by admitting local finite quotient singularities \cite{MR79769}. Locally, it is modelled on quotients $U/G$, where $G$ is a finite group acting holomorphically and effectively on an open subset $U \subseteq \mathbb{C}^n$. In the algebraic category, the geometric data of a complex orbifold formally corresponds to a smooth, separated Deligne--Mumford stack over $\mathbb{C}$ with generically trivial stabilisers \cite{MR2778793}. Throughout this paper, we treat these two notions as equivalent.

Chen and Ruan \cite{MR2104605} introduced a cohomology theory for orbifolds based on the \emph{inertia stack} $I\cX$. In the present setup, the inertia stack decomposes as a disjoint union indexed by the conjugacy classes of $G$:
\beq
\label{def: inertia_stack}
  I\cX \;\cong\; \bigsqcup_{(g)\in {\rm Conj}(G)} \cX_{(g)},
  \qquad
  \cX_{(g)} \coloneqq [Y^g / C(g)],
\eeq
where $Y^g$ is the fixed locus and $C(g)$ is the centraliser of $g$.

The connected components of the substacks $\cX_{(g)}$ are called \emph{sectors}. Those contained in $\cX_{(1)}$ form the \textbf{untwisted} sector, and the remaining ones are \textbf{twisted} sectors.

There is a canonical involution
\beq
  \operatorname{inv}\colon I\cX \to I\cX, \qquad (x,g) \mapsto (x,g^{-1}),
\eeq
which exchanges $\cX_{(g)}$ with $\cX_{(g^{-1})}$.

\smallskip
\noindent\textbf{Age.}
For $x\in Y^g$, the element $g$ acts on $T_xY$ with eigenvalues $\exp(2\pi \ri\,\theta_j)$, where $\theta_j\in[0,1)$.
Define
\beq
  \age(g)\big|_{x} \coloneqq \sum_j \theta_j.
\eeq
On each connected component of $Y^g$ this is constant, so we view $\age(g)$ as a locally constant function on $\cX_{(g)}$.

\noindent\textbf{State space.}
The Chen--Ruan cohomology of $\cX$ is
\beq
  H^\ast_{\mathrm{CR}}(\cX) \coloneqq H^\ast(I\cX),
\eeq
equipped with the Chen--Ruan grading by declaring that a class
$\alpha \in H^k(\cX_{(g)})$ has shifted degree
\beq
  \deg_{\mathrm{CR}}(\alpha) \coloneqq k + 2\,\age(g).
\eeq
Equivalently, as a graded vector space one may write
\beq
  H^\ast_{\mathrm{CR}}(\cX)
  \;\cong\;
  \bigoplus_{(g)} H^{\ast - 2\,\age(g)}(\cX_{(g)}).
\eeq

We call the contribution from twisted sectors the \emph{twisted part} of $H^\ast_{\mathrm{CR}}(\cX)$, namely

\beq
  H^{\ast,{\rm tw}}_{\mathrm{CR}}(\cX)
  \; :=\;
  \bigoplus_{(g)\neq (1)} H^{\ast - 2\,\age(g)}(\cX_{(g)}).
\eeq

\smallskip
\noindent\textbf{Orbifold Poincar\'e pairing.}
For the remainder of this subsection, assume that $\cX$ is proper.
The orbifold Poincar\'e pairing on $\cX$ is defined using the standard pairing and the involution on $I\cX$:
\beq
  (\alpha,\beta)_{\mathrm{CR}}
  \;\coloneqq\;
  \int_{I\cX} \alpha \cup \operatorname{inv}^\ast(\beta),
  \qquad \alpha,\beta \in H^\ast_{\mathrm{CR}}(\cX),
\eeq
where integration over the component $\cX_{(g)}=[Y^g/C(g)]$ is understood as
\beq
  \int_{[Y^g/C(g)]} (\cdot) \;\coloneqq\; \frac{1}{|C(g)|} \int_{Y^g} (\cdot)
\eeq
whenever this interpretation is convenient.

\smallskip
\noindent\textbf{Double inertia stack and the obstruction bundle.}
The \emph{double inertia stack} (or 2-fold inertia) of $\cX$ is defined by
\beq
  I_2\cX \;\coloneqq\; I\cX \times_{\cX} I\cX,
\eeq
with the two projections
${\rm e}_1,{\rm e}_2\colon I_2\cX \to I\cX$.
There is also a natural morphism
${\rm \mu}\colon I_2\cX\to I\cX$ induced by multiplying stabiliser elements.

Set
\beq
  Y^{g,h}\coloneqq Y^g\cap Y^h,\qquad
  C(g,h)\coloneqq C(g)\cap C(h),\qquad
  \cX^{g,h}\coloneqq [Y^{g,h}/C(g,h)].
\eeq
Then $I_2\cX$ admits the presentation
\beq
  I_2\cX
  \;\cong\;
  \Bigl[\coprod_{(g,h)\in G\times G} Y^{g,h}\,\Big/\,G\Bigr]
  \;\cong\;
  \coprod_{[(g,h)]\in G\backslash(G\times G)} \cX^{g,h},
\eeq
where $G$ acts on $G\times G$ by simultaneous conjugation.
On the $(g,h)$-component, the morphisms
\beq
  {\rm e}_1:\cX^{g,h}\to \cX_{(g)},\qquad
  {\rm e}_2:\cX^{g,h}\to \cX_{(h)},\qquad
  {\rm \mu}:\cX^{g,h}\to \cX_{(gh)}
\eeq
are induced by the inclusions
$Y^{g,h}\hookrightarrow Y^g$, $Y^{g,h}\hookrightarrow Y^h$, and
$Y^{g,h}\hookrightarrow Y^{gh}$, together with the corresponding inclusions of centralisers. Equivalently, the multiplication map is induced by
$(y,g,h)\mapsto (y,gh)$ on stabilisers.

There is a canonically defined obstruction bundle $\mathrm{Ob}_{g,h}$ on each component $\mathcal{X}^{g,h} \subset I_2\mathcal{X}$. Its Euler class
\begin{equation}
e_{g,h}:=e(\mathrm{Ob}_{g,h}) \in H^{2r_{g,h}}(\mathcal{X}^{g,h})
\end{equation}
plays a central role in the definition of the orbifold (Chen--Ruan) cup product; see \cite{MR2104605,MR1971293} for its construction and properties. 

\smallskip
\noindent\textbf{Orbifold cup product.}
For the following componentwise formula only, we assume that $G$ is abelian; this includes all groups appearing in the folding constructions of this paper.
Let $\alpha\in H^*(\cX_{(g)})$ and $\beta\in H^*(\cX_{(h)})$ be homogeneous.
Their Chen--Ruan product $\alpha\star_{\rm CR}\beta$ is the class in the $(gh)$-sector defined by
\beq
  \alpha\star_{\rm CR}\beta
  \;\coloneqq\;
  {\rm \mu}_*\Bigl({\rm e}_1^*\alpha \;\cup\; {\rm e}_2^*\beta \;\cup\; e_{g,h}\Bigr)
  \;\in\; H^*(\cX_{(gh)}).
  \label{eq:CR_product}
\eeq
Taking the direct sum over $g\in G$ and inserting the degree shifts yields a graded product on $H^*_{\rm CR}(\cX)$.
For a nonabelian group, the product is obtained by summing the same push--pull construction over the simultaneous-conjugacy classes of pairs $(g',h')$ with $g'\in(g)$ and $h'\in(h)$.

The Chen--Ruan product, together with the orbifold Poincar\'e pairing, defines the \emph{degree-zero three point function} for any $\alpha,\beta,\gamma\in H^\ast_{\mathrm{CR}}(\cX) $:
\beq
\langle\alpha,\beta,\gamma\rangle_0  
 \;\coloneqq\;
 (\alpha \star_{\rm CR} \beta,\gamma)_{\mathrm{CR}}.
 \label{eq: degree zero 3pt func}
\eeq

\subsection{Equivariant variants of Chen--Ruan cohomology}
\label{subsec:equivariant-variants}

Chen--Ruan cohomology admits a natural $\bbC^\times$-equivariant refinement. In this paper we work primarily in this $\bbC^\times$-equivariant setting.

\smallskip
\noindent\textbf{$\bbC^\times$-equivariant Chen--Ruan cohomology.}
We set
\beq
  H^\ast_{\mathrm{CR},\bbC^\times}(\cX) \coloneqq H^\ast_{\bbC^\times}(I\cX),
\eeq
equipped with the same Chen--Ruan grading convention as in \cref{subsec:CR}.
The equivariant orbifold Poincar\'e pairing is
\beq
  (\alpha,\beta)_{\mathrm{CR},\bbC^\times}
  \;\coloneqq\;
  \int_{I\cX}^{\bbC^\times} \alpha \cup \operatorname{inv}^\ast(\beta)
  \in \bbQ(\nu),
\eeq
where $\int^{\bbC^\times}$ denotes the $\bbC^\times$-equivariant pushforward to a point.

\begin{rmk}[Non-proper targets]\label{rem:nonproper-equivariant}
Since $Y$ need not be proper, the non-equivariant pushforward
$\int_{I\cX}$ need not be defined. In this paper all pairings and correlators are understood
$\bbC^\times$-equivariantly, and their values lie in the localised coefficient ring $\bbQ(\nu)$.
The hypothesis that $Y^{\bbC^\times}$ is compact ensures that these localised pushforwards are well-defined via localisation.
\end{rmk}

\smallskip
\noindent\textbf{$\bbC^\times$-equivariant Chen--Ruan product.}
The Chen--Ruan product admits a $\bbC^\times$-equivariant refinement
\beq
  \star_{\mathrm{CR},\bbC^\times}\;:\;
  H^\ast_{\mathrm{CR},\bbC^\times}(\cX)\otimes H^\ast_{\mathrm{CR},\bbC^\times}(\cX)
  \longrightarrow H^\ast_{\mathrm{CR},\bbC^\times}(\cX),
\eeq
defined by the same correspondence on the double inertia stack as in \cref{subsec:CR},
with all pullbacks, pushforwards, and Chern classes taken $\bbC^\times$-equivariantly.
Accordingly, the degree--zero three--point function~\eqref{eq: degree zero 3pt func} admits a tautological $\bbC^\times$-equivariant extension.
\subsection{$\bbC^\times$-equivariant Gromov--Witten invariants and quantum cohomology}
\label{sec:GW-QC}

In this subsection, we briefly recall the necessary background on equivariant Gromov--Witten theory for both smooth varieties and orbifolds. Our main references are \cite{MR2104605,MR2450211,MR1950941}.

\subsubsection{$\bbC^\times$-equivariant Gromov--Witten invariants of smooth varieties}
\label{subsec:smoothGW}

Let $Z$ be a smooth quasi-projective variety equipped with a $\bbC^\times$-action.
Let $\overline{M}_{g,n}(Z,\beta)$ denote the moduli space of genus-$g$, $n$-pointed stable maps to $Z$ of class $\beta \in H_2(Z, \bbZ)$. By the standard theory of Gromov--Witten invariants, it admits a perfect obstruction theory and a virtual fundamental class
\[
  [\overline{M}_{g,n}(Z,\beta)]^{\vir} \in A_{\mathrm{vdim}}\bigl(\overline{M}_{g,n}(Z,\beta)\bigr),
\]
where the virtual dimension is given by
\begin{equation}\label{eq:vdim-smooth}
  \mathrm{vdim}\,\overline{M}_{g,n}(Z,\beta) = (1-g)(\dim Z - 3) + n + \int_{\beta} c_1(T_Z).
\end{equation}

The $\bbC^\times$-action on $Z$ naturally induces a $\bbC^\times$-action on the moduli space $\overline{M}_{g,n}(Z,\beta)$, and the virtual fundamental class has a natural equivariant refinement
\[
  [\overline{M}_{g,n}(Z,\beta)]^{\vir,\bbC^\times} \in A_{\mathrm{vdim}}^{\bbC^\times}\bigl(\overline{M}_{g,n}(Z,\beta)\bigr).
\]

Gromov--Witten invariants are intersection numbers against the virtual fundamental class. In the $\bbC^\times$-equivariant setup, one needs to use Graber--Pandharipande's virtual localisation theorem \cite{MR1666787} to define such invariants for a noncompact target. For every $(g,n,\beta)$ under consideration, assume that the fixed locus of the induced $\bbC^\times$-action on $\overline{M}_{g,n}(Z,\beta)$ is proper. Virtual localisation then defines, for insertions $\alpha_1, \dots, \alpha_n \in H^\ast_{\bbC^\times}(Z)$, the \emph{primary $\bbC^\times$-equivariant Gromov--Witten invariants}:
\begin{equation}\label{eq:equiv-gw-smooth}
  \big\langle \alpha_1,\dots,\alpha_n \big\rangle^{Z,\bbC^\times}_{g,n,\beta}
  \;\coloneqq\;
  \sum_{F \subset \overline{M}_{g,n}(Z,\beta)^{\bbC^\times}}
  \int_{[F]^{\vir}}
  \frac{i_F^\ast\!\Bigl(\prod_{i=1}^n {\rm ev}_i^\ast(\alpha_i)\Bigr)}{e_{\bbC^\times}(N^{\vir}_F)}
  \;\in\; \bbQ(\nu),
\end{equation}
where ${\rm ev}_i \colon \overline{M}_{g,n}(Z,\beta) \longrightarrow Z$ is the evaluation map, the sum runs over the connected components $F$ of the fixed locus, $i_F\colon F\hookrightarrow \overline{M}_{g,n}(Z,\beta)$ is the inclusion, and the Euler class of the virtual normal bundle $N^{\vir}_F$ is inverted in the localised equivariant cohomology of a point.

\subsubsection{$\bbC^\times$-equivariant Gromov--Witten invariants of orbifolds}
\label{subsec:orbGW}

Gromov--Witten invariants are also defined for orbifolds. This theory was first developed by Chen--Ruan \cite{MR1950941} in the symplectic category and later by Abramovich--Graber--Vistoli \cite{MR2450211} in the algebraic category. In this subsubsection, $\cX$ may be any smooth orbifold, not necessarily a global quotient stack. Assume that $\cX$ is equipped with a $\bbC^\times$-action. Following \cite{MR2450211}, one considers the moduli stack $\cK_{g,n}(\cX,\beta)$ of genus-$g$, $n$-pointed twisted stable maps to $\cX$ of degree $\beta$. It carries a canonical perfect obstruction theory and hence a virtual fundamental class
\[
  [\cK_{g,n}(\cX,\beta)]^{\vir}\in A_{\mathrm{vdim}}\bigl(\cK_{g,n}(\cX,\beta)\bigr).
\]
Here, $\mathrm{vdim}$ denotes the virtual dimension on a connected component: if the $i$-th marking lies in the twisted sector of $I\cX$ of age $\age_i$ $(1\le i\le n)$, then
\begin{equation}\label{eq:vdim-orbifold}
  \mathrm{vdim}\,\cK_{g,n}(\cX,\beta)
  \;=\;
  (1-g)\bigl(\dim \cX-3\bigr)\;+\;n\;+\;\int_{\beta} c_1(T_{\cX})
  \;-\;\sum_{i=1}^n \age_i .
\end{equation}

For each marking, we use the evaluation pullback
\beq
{\rm ev}_i^*\colon
H^*_{\rm CR,\bbC^\times}(\cX)
\longrightarrow
H^*_{\bbC^\times}\bigl(\cK_{g,n}(\cX,\beta)\bigr).
\eeq
We refer to \cite[\S 4.4]{MR2450211} and \cite[\S 2]{MR3414388} for the construction and conventions.

Just as in the smooth case, for every $(g,n,\beta)$ under consideration, assume that the fixed locus of the induced $\bbC^\times$-action on $\cK_{g,n}(\cX,\beta)$ is proper. The virtual localisation theorem \cite{MR1666787} then defines, for insertions $\alpha_i\in H^\ast_{\mathrm{CR},\bbC^\times}(\cX)=H^\ast_{\bbC^\times}(I\cX)$, the corresponding \emph{primary $\bbC^\times$-equivariant orbifold Gromov--Witten invariants}

\beq\label{eq:equiv-gw}
  \big\langle \alpha_1,\dots,\alpha_n \big\rangle^{\cX,\bbC^\times}_{g,n,\beta}
  \;\coloneqq\;
  \sum_{F \subset \cK_{g,n}(\cX,\beta)^{\bbC^\times}}
  \int_{[F]^{\vir}}
  \frac{i_F^\ast\!\Bigl(\prod_{i=1}^n {\rm ev}_i^\ast(\alpha_i)\Bigr)}{e_{\bbC^\times}(N^{\vir}_F)}
  \;\in\; \bbQ(\nu).
\eeq

\begin{rmk}\label{rem:CY3-vdim}
If $\cX$ is a Calabi--Yau threefold, i.e.\ $\dim \cX=3$ and $\omega_{\cX}\cong \cO_{\cX}$ (equivalently $c_1(T_{\cX})=0$), then \eqref{eq:vdim-orbifold} reduces to
\[
  \mathrm{vdim}\,\cK_{g,n}(\cX,\beta)=n-\sum_{i=1}^n \age_i,
\]
and in particular $\mathrm{vdim}\,\cK_{g,0}(\cX,\beta)=0$.
\end{rmk}

\subsubsection{$\bbC^\times$-equivariant quantum cohomology}
\label{subsec:orbQH}

To treat the smooth and orbifold cases uniformly, let $\cZ$ denote either a smooth quasi-projective variety or an orbifold. We establish the convention that $H(\cZ)$ refers to the standard equivariant cohomology $H^\ast_{\bbC^\times}(\cZ)$ for the smooth case, and the equivariant Chen--Ruan cohomology $H^\ast_{\mathrm{CR},\bbC^\times}(\cZ)$ for the orbifold case. Correspondingly, $\langle \cdots \rangle^{\cZ,\bbC^\times}_{g,n,\beta}$ indicates the equivariant Gromov--Witten invariants defined in \eqref{eq:equiv-gw-smooth} or \eqref{eq:equiv-gw}.

Let $\tau$ be a formal parameter in $H(\cZ)$. All $\bbC^\times$-equivariant products below are taken after localising the equivariant parameter, so the resulting Frobenius structures are defined over $\bbC(\nu)$.
The \emph{$\bbC^\times$-equivariant big quantum product} $*_\tau$ on $H(\cZ)$ is defined by
\beq
  (\alpha_1 *_\tau \alpha_2,\alpha_3)_{\cZ,\bbC^\times}
  \;\coloneqq\;
  \sum_{\beta} \sum_{n\ge 0} \frac{1}{n!}
  \big\langle \alpha_1,\alpha_2,\alpha_3,\tau,\dots,\tau \big\rangle^{\cZ,\bbC^\times}_{0,n+3,\beta},
  \label{eq: big_qp}
\eeq
where the pairing on the left is the standard equivariant Poincar\'e pairing (or the equivariant orbifold Poincar\'e pairing, respectively).

The \emph{$\bbC^\times$-equivariant quantum cohomology} is defined as the family
\beq
  \mathrm{QH}^\ast_{\bbC^\times}(\cZ)
  \;\coloneqq\;
  \big(H(\cZ), *_\tau, \langle,\rangle, \bm{1}\big),
\eeq
viewed as a (formal) Frobenius manifold structure over $\bbC(\nu)$.

\begin{rmk}[Fundamental class axiom]\label{rem:fund_class_axiom}
The \textbf{fundamental class axiom} for Gromov--Witten invariants says that
\[
\big\langle \alpha_1,\dots,\alpha_n,\bm{1}\big\rangle^{\cZ,\bbC^\times}_{g,n+1,\beta}
= 0,
\]
except for the case $(g,n+1,\beta) = (0,3,0)$, where 
\[
\big\langle \alpha ,\dots,\beta,\bm{1}\big\rangle^{\cZ,\bbC^\times}_{0,3,0}
= \int_{\cZ}^{\bbC^\times}\alpha\cup\beta.
\]
This implies that $\bm{1}$ is the unit of the Frobenius algebra.
\end{rmk}

\begin{rmk}[Divisor equation]\label{rem:divisor-equation}
For a divisor class $D\in H^2_{\bbC^\times}(\cZ)$ and an effective curve class $\beta\neq 0$, the \textbf{divisor equation} for Gromov--Witten invariants implies that
\[
\big\langle \alpha_1,\dots,\alpha_n,D\big\rangle^{\cZ,\bbC^\times}_{0,n+1,\beta}
=\int_\beta D\,
\big\langle \alpha_1,\dots,\alpha_n\big\rangle^{\cZ,\bbC^\times}_{0,n,\beta}.
\]
Thus repeated divisor insertions in the big quantum product exponentiate to the factor $e^{\langle\beta,\tau'\rangle}$ below. 
\end{rmk}

\begin{rmk}[Small quantum cohomology and Novikov parameters]\label{rem:quantum-params}
The product defined above is the \emph{big} quantum product, since it depends on a general parameter $\tau\in H(\cZ)$.

The $\bbC^\times$-equivariant \emph{small} quantum cohomology keeps only the dependence on $H^2_{\bbC^\times}(\cZ)$, and the curve classes are encoded by Novikov variables. Concretely, restricting to $\tau'\in H^2_{\bbC^\times}(\cZ)$ and applying the divisor equation, \eqref{eq: big_qp} becomes, after resummation:
\[
(\alpha_1 *_{\tau'} \alpha_2,\alpha_3)_{\cZ,\bbC^\times} = \sum_{\beta}  e^{\int_\beta \tau'}
  \big\langle \alpha_1 ,\alpha_2,\alpha_3\big\rangle^{\cZ,\bbC^\times}_{0,3,\beta}.
\]
The $e^{\int_\beta \tau'}$ coefficients are often replaced by \textbf{Novikov parameters} $Q^\beta$. Let ${\rm Eff}(\cZ)\subset H_2(\cZ;\bbZ)$ be the semigroup of effective curve classes, choose generators $\beta_1,\dots,\beta_r$ of ${\rm Eff}(\cZ)$, and let
\[
  \Lambda_{\mathrm{Nov}}\coloneqq \bbQ[[Q_1,\dots,Q_r]]
\]
be the Novikov ring. For $\beta=\sum_{i=1}^r d_i\beta_i$ we write
\[
  Q^\beta \coloneqq \prod_{i=1}^r Q_i^{d_i}\in \Lambda_{\mathrm{Nov}}.
\]
The $\bbC^\times$-equivariant \emph{small} quantum product $\star_{\mathrm{sm}}$ on $H(\cZ)\otimes \Lambda_{\mathrm{Nov}}$ is defined by
\beq
  (\alpha_1\star_{\mathrm{sm}}\alpha_2,\alpha_3)_{\cZ,\bbC^\times}
  \;=\;
  \sum_{\beta\in {\rm Eff}(\cZ)} Q^\beta\,
  \bigl\langle \alpha_1,\alpha_2,\alpha_3 \bigr\rangle^{\cZ,\bbC^\times}_{0,3,\beta}.
  \label{eq:small_qp}
\eeq

If Gromov--Witten invariants can be defined for $\cZ$ without localisation (e.g., if $\cZ$ is proper), the same construction applies and defines the non-equivariant small quantum cohomology on the standard (or non-equivariant Chen--Ruan) cohomology spaces.
\end{rmk}

\section{Proof of \cref{thm: qproduct_main_thm}}
\label{sec:proof_main_theorem}
\subsection{Proof of \cref{thm: qproduct_main_thm} assuming \Cref{properties: threefolds} and \cref{prop: multicover_formula}}
\label{subsec: sketch_of_proof}
We prove \cref{thm: qproduct_main_thm} in this section. The argument relies on \cref{properties: threefolds} and \cref{prop: multicover_formula}, which are established in \cref{subsec: threefolds_with_properties} and \cref{subsec: multi_cover_formula} respectively.

We first describe the Chen--Ruan cohomology of $\cX_{\RR^{\mathrm{fold}}}$. By Slodowy \cite{MR584445}, the $\Phi_R$-action on $Z_\RR$ has $N_{\RR^{\mathrm{fold}}}$ fixed points $\{p_k\}_{k=1}^{N_{\RR^{\mathrm{fold}}}}$. Therefore, the inertia stack decomposes as
\beq
I\cX_{\RR^{\mathrm{fold}}} = \cX_{\RR^{\mathrm{fold}}}\bigsqcup_{k=1}^{N_{\RR^{\mathrm{fold}}}}\sqcup_{i=1}^{|\Phi_{\RR}|-1} [p_k/\Phi_{\RR}] ,
\eeq
and
\beq
\label{eq: CRcohomology}
H^*_{\rm CR}(\cX_{\RR^{\mathrm{fold}}};\bbC)\cong H^*(\cX_{\RR^{\mathrm{fold}}};\bbC)\bigoplus_{k=1}^{N_{\RR^{\mathrm{fold}}}}\oplus_{i=1}^{|\Phi_{\RR}|-1} H^{*-2}(B\bbZ_{|\Phi_\RR|-1};\bbC).
\eeq

\begin{prop}
\label{prop: untwisted_part_closed}
    $ \forall \phi,\phi',\tau\in H^2(\cX_{\RR^{\mathrm{fold}}};\bbC)$ and $\delta\in H^{\rm tw}_{\rm CR}(\cX_{\RR^{\mathrm{fold}}};\bbC)$,
    \beq
    (\phi\star_{\tau}\phi',\delta)_{\mathrm{CR},\bbC^\times}= 0, 
    \eeq
\end{prop}
\begin{proof}

    In orbifold Gromov-Witten theory for global quotients (see for example \cite[(10)]{MR2148194}), the product of monodromies at marked point need to be $1$ in order to have non-empty moduli space, otherwise the GW invariant is zero. As a consequence, each term
    \beq
    \langle\phi,\phi',\tau \cdots,\tau,\delta\rangle_{0,n+3,A},
    \eeq
    in the expansion of quantum product is zero, since the product of monodromies is a nontrivial $|\Phi_\RR|$-th root of unity.
\end{proof}

As a corollary, the quantum product is closed on the subspace 
\beq H^*(\cX_{\RR^{\mathrm{fold}}};\bbC)\subset H^*_{\rm CR}(\cX_{\RR^{\mathrm{fold}}};\bbC).\eeq

We start with computations of several genus zero Gromov-Witten invariants of $\cX_{\RR^{\mathrm{fold}}}$.
\begin{prop}
    For any $\phi_i,\phi_j,\phi_k\in H^2(\cX_{\RR^{\mathrm{fold}}};\bbC) \cong \mathfrak{h}_{\RR^{\mathrm{fold}}}$, the degree $0$ invariants are given by
    \begin{enumerate}
        \item[(1)] $\langle1,1,1\rangle_0  =\frac{1}{|\Phi_\RR|} \frac{1}{t^2 |G_\RR|}$;
        \item[(2)] $\langle\phi_i,1,1\rangle_0 = 0$;
        \item[(3)] $\langle \phi_i,\phi_j,1\rangle_0 = \frac{1}{|\Phi_\RR|}\langle \phi_i,\phi_j\rangle$;
        \item[(4)] $ \langle\phi_i,\phi_j,\phi_k\rangle_0 = \frac{-\nu}{|\Phi_\RR|}  \sum\limits_{\overline{\beta}\in \RR^\mathrm{ave,+}}\langle\phi_i,\overline{\beta}\rangle \langle\phi_j,\overline{\beta}\rangle \langle\phi_k,\overline{\beta}^{\vee}\rangle $.
    \end{enumerate}
    \label{prop: qproduct_degree0}
\end{prop}
\begin{proof}
By the comparison \eqref{eq: second_cohomology} of Poincar\'e pairings, the degree $0$ of $\cX_{\RR^{\mathrm{fold}}}$ satisfy
    \beq
    \langle\cdot,\cdot,\cdot\rangle^{\cX_{\RR^{\mathrm{fold}}}}_0 = \frac{1}{|\Phi_\RR|}\langle\cdot,\cdot,\cdot\rangle^{Z_\RR}_0.
    \eeq  
The later ones are computed in \cite{MR2411404} and $(1)-(3)$ immediately follows. By \cite{MR2411404},
    \beq
     \langle\phi_i,\phi_j,\phi_k\rangle^{Z_\RR}_0 = -\nu\sum\limits_{\beta\in \RR^+}\langle\phi_i,\beta\rangle \langle\phi_j,\beta\rangle \langle\phi_k,\beta\rangle,
    \eeq
(4) then follows from  the following properties:
\begin{itemize}
    \item $\langle \phi,\beta\rangle = \langle \phi,\overline{\beta}\rangle$ for any $\phi\in \mathfrak{h}_{R^{\mathrm{fold}}} = \mathfrak{h}_{\RR}^{\Phi_\RR}$;
    \item $\sum\limits_{g\in \Phi_\RR}g\circ \beta = \overline{\beta}^{\vee}$, shown in \cite[Lemma 9]{MR2411404}.
\end{itemize}
\end{proof}
\begin{thm} For any $A\in H_2(\cX_{\RR^{\mathrm{fold}}};\bbZ)\cong \Lambda_{\RR^{\mathrm{ave}}}$, the degree $A$ invariants are given by
\[
\langle\rangle^{\cX_{\RR^{\mathrm{fold}}}}_A=
\begin{cases}
2\nu \frac{1}{d^3}\frac{2}{\langle\bar\beta,\bar\beta\rangle}\frac{1}{|\Phi_\RR|}, & A = d\bar\beta,\\
0,  & \text{otherwise}.
\end{cases}
\]
\label{thm: qproduct_quantum}
\end{thm}
\begin{outlineproof}
In \cref{def: threefolds}, we will define two three-folds $W$ and $W'$ satisfying:
\begin{property}
    \begin{enumerate}
    \item[(P1)] $W$ and $W'$ admit $\Phi_{\RR}$-actions;
    \item[(P2)] $W$ is deformation equivalent to $W'$ ;
    \item[(P3)] The $\Phi_{\RR}$-action is compatible with the embedding $Z_\RR\subset W$ and deformation;
    \item[(P4)] $Z_\RR\subset W$ and contains all compact curves of $W$. There is a contracting map $W\to W^{\rm aff}$ where all the curves get contracted and $W^{\rm aff}$ is affine;
    \item[(P5)] Given any positive root $\beta\in \RR$, there exists an isolated $(-1,-1)$ curve $C_{\beta}\subset W'$. Moreover, $\{C_{\beta}:\beta\in \RR^{+}\}$ exhaust all the compact curves in $W'$ and $\Phi_{\RR}$ acts on $\{C_{\beta}:\beta\in \RR^{+}\}$ respecting the $\Phi_{\RR}$-action on the index $\beta$;
    \item[(P6)]  On $\cN_{Z_\RR\setminus W}$, $\Phi_{\RR}$ acts trivially. On $W'$, the stabliser of $\Phi_{\RR}$ is either trivial or $\Phi_{\RR}$. At a $\Phi_{\RR}$ fixed point on $C_{\beta}$, $\Phi_\RR$ acts with weights $1$ (or $-1$) on tangent direction on $C_{\beta}$, and $(-1,0)$ (or $(1,0)$) on the normal direction;
    \item[(P7)] $W$ and $W'$ are Calabi--Yau, and the $\Phi_{\RR}$-actions preserve the Calabi--Yau form on them.     
\end{enumerate}
\label{properties: threefolds}
\end{property}	

These properties above immediately imply that $\langle\rangle_A^{\cX_{\RR^{\mathrm{fold}}}} = 2\nu\langle\rangle^{[W/\Phi_{\RR}]}_A = 2\nu\langle\rangle^{[W'/\Phi_{\RR}]}_A$. The first identity follows from the fact that $\Phi_{\RR}$ acts trivially on $\cN_{Z_\RR\setminus W}$, and $\cN_{Z_\RR\setminus W}\cong \mathcal{O}_{Z_\RR}\otimes \bbC_{2\nu}$ as a $\bbC^{\times}$ bundle. The second one follows from deformation invariance of orbifold Gromov-Witten invariants.

By \emph{(P4)}, for any genus-zero stable map to $[W'/\Phi_{\RR}]$ in class $d\beta$ the image is \emph{equal} to $C_\beta$.
Take
\[
H_{\beta} \;:=\; \mathrm{Stab}_{\Phi_{\RR}}(C_\beta)
\;=\;
\{\, g\in \Phi_{\RR}\mid g(C_\beta)=C_\beta \,\}.
\]
Since $C_\beta$ is a $(-1,-1)$-curve and does not intersect with other $C_{\beta},\beta'\neq \beta$, the degree-$d\beta$ invariant is the local contribution supported on $C_\beta$, namely
\begin{equation}\label{eq:local_contribution}
\big\langle\,\big\rangle^{[W'/\Phi_{\RR}]}_{0,d\beta}
\;=\;
\int_{[\overline \cK_{0,0}([C_\beta/H_\beta],d)]^{\vir}}
e\!\left(R^1\pi_* f^* N_{C_\beta/W'}\right),
\end{equation}
where $\pi:\mathcal{C}\to \overline M_{0,0}([C_\beta/H_{\beta}],d)$ and $f:\mathcal{C}\to [C_\beta/H_{\beta}]$ are the universal curve and map.

A direct check shows that
\begin{equation}\label{eq:long_roots}
|H_{\beta}| \;=\; \frac{\langle\bar\beta,\bar\beta\rangle}{2}\cdot |\Phi_\RR|.
\end{equation}
The theorem therefore reduces to the following multicover formula:

\begin{prop}[=\cref{cor: multi_cover_formula}]\label{prop: multicover_formula}
Let $x,y$ be homogeneous coordinates of $\bbP^1$, and let $\zeta$ be a primitive root of unity of order $n$. Consider the $\bbZ_n=\langle\zeta\rangle$-action on $\cO_{\bbP^1}(-1)\oplus \cO_{\bbP^1}(-1)$ as follows:
\begin{enumerate}
    \item[(1)] On $\bbP^1$, $\zeta[x:y] = [\zeta x:y]$;\\
    \item[(2)] View $x,y$ as generators of $H^0(\cO_{\bbP^1}(1))$, then on the first (resp. second) summand $\cO_{\bbP^1}(-1)$, $\zeta$ has weight $0$ (resp. $1$) in $x$ and weight $-1$(resp. $0$) in $y$.
\end{enumerate}
Let $[\Tot(\cO_{\bbP^1}(-1)\oplus \cO_{\bbP^1}(-1))/\bbZ_n]$ be the corresponding quotient stack/orbifold, then
\begin{equation}\label{eq:multicover_value}
\big\langle\,\big\rangle^{[\Tot(\cO_{\bbP^1}(-1)\oplus \cO_{\bbP^1}(-1))/\bbZ_n]}_{0,d}
\;=\;
\frac{1}{n\,d^3}.
\end{equation}
\end{prop}
\end{outlineproof}

\begin{proof}[Proof of \cref{thm: qproduct_main_thm} assuming \Cref{properties: threefolds} and \cref{prop: multicover_formula}]
From \cref{prop: untwisted_part_closed}, the orbifold quantum product is closed on the untwisted sector. Therefore, to compute  $\phi_i\star_\tau\phi_j$, it suffices to compute its pair with $H^*(\cX_{\RR^{\mathrm{fold}}};\bbC)$. By \cref{rem:fund_class_axiom} the pairing with $\bm{1}$ is just the classical pairing , which has been shown in \cref{prop: qproduct_degree0}. Therefore, it remains to show, for any $\phi_i,\phi_j,\phi_k,\tau\in H^2(\cX_{\RR^{\mathrm{fold}}};\bbC)$, it holds
\beq
(\phi_i\star_\tau\phi_j,\phi_k)
=
\frac{-\nu}{|\Phi_\RR|}\sum_{\overline{\beta}\in \RR^\mathrm{ave,+}}\langle\overline{\beta},\phi_i\rangle \langle\overline{\beta},\phi_j\rangle \langle\overline{\beta}^{\vee},\phi_j\rangle \frac{1+\re^{-\langle\bar\beta,\tau\rangle}}
{1-\re^{-\langle\bar\beta,\tau\rangle}}.
\eeq

Recall that
\beq
\label{eq: qprod_3pt_func}
(\phi_i\star_\tau\phi_j,\phi_k)
=
\langle\phi_i,\phi_j,\phi_k\rangle_0
+
\sum_{A\neq0}
\re^{\int_A\tau}
\int_A\phi_i
\int_A\phi_j
\int_A\phi_k
\,\langle\,\rangle_A.
\eeq
Here the first term is the degree-zero contribution, which is computed in \cref{prop: qproduct_degree0}. The remaining terms are determined by the positive-degree no-point invariants. By \Cref{thm: qproduct_quantum}, they vanish unless $A=d\bar\beta$, where $\bar\beta\in\RR^{\mathrm{ave},+}$ and $d\geq1$. Under the root-lattice identifications,
\beq
\int_{d\bar\beta}\phi=-d\langle\bar\beta,\phi\rangle.
\eeq
Thus the three divisor pairings contribute a factor $-d^3$. The
formula in \cref{thm: qproduct_quantum} contains a factor
$d^{-3}$, so these powers of $d$ will cancel. For each such root, the
remaining coefficient will give
\beq
-2\nu \frac{2}{\langle\bar\beta,\bar\beta\rangle}\frac{1}{|\Phi_\RR|}\sum_{d\geq1}\re^{-d\langle\bar\beta,\tau\rangle}
=
-2\nu \frac{2}{\langle\bar\beta,\bar\beta\rangle}\frac{1}{|\Phi_\RR|}\frac{\re^{-\langle\bar\beta,\tau\rangle}}
{1-\re^{-\langle\bar\beta,\tau\rangle}}.
\eeq
Notice that
\beq
\bar{\beta}^\vee = \frac{2\bar{\beta}}{\langle\bar{\beta},\bar{\beta}\rangle},
\eeq
combining this series with the degree-zero contribution gives the factor
\beq
\frac{-\nu}{|\Phi_\RR|}\frac{1+\re^{-\langle\bar\beta,\tau\rangle}}
{1-\re^{-\langle\bar\beta,\tau\rangle}}
\eeq
in \eqref{eq: qprod_3pt_func}. This completes the proof.
\end{proof}

\subsection{Construction of $W$ and $W'$ with \Cref{properties: threefolds}}
\label{subsec: threefolds_with_properties}
Restricting Grothendieck's simultaneous resolution
\[
\xymatrix@R=3.2em@C=4.2em{
\widetilde{\mathfrak g} \ar[r]^{\pi} \ar[dr]_{\widetilde{\chi}} &
\mathfrak g\times_{\mathfrak h/W}\mathfrak h \ar[r] \ar[d] &
\mathfrak g \ar[d]^{\chi} \\
& \mathfrak h \ar[r] & \mathfrak h/W
}
\]
to the Slodowy slice through a subregular nilpotent element yields
\begin{equation}
\xymatrix@R=3.2em@C=4.2em{
\widetilde S \ar[r]^{\pi_S} \ar[dr]_{\widetilde\chi_S} &
S\times_{\mathfrak h/W}\mathfrak h \ar[d]^{\chi_S} \\
& \mathfrak h
}
\label{diagram: simultaneous_resolution_S}
\end{equation}

\begin{property}[\cite{MR584445}]
The objects in the diagram \eqref{diagram: simultaneous_resolution_S} have the following properties:
    \begin{itemize}
    \item over $0\in \mathfrak{h}$, $\pi_S$ is the minimal resolution of the corresponding Kleinian singularity;
    \item $\chi^{-1}(\tau)$ is smooth, and $\widetilde\chi_S^{-1}(\tau)$ contains a curve only when $\tau\in H_{\beta}$ for some root $\beta$;
    \item Let $\tau\in \cup_{\beta}H_{\beta}$, then $\chi_S^{-1}(\tau)$ has the singularity type determined by the root subsystem 
    \beq\{\alpha\in \RR,\tau \in H_{\alpha}\};\eeq
    \item The symmetric group  $\Phi_{\RR}$ acts everywhere on the resolution diagram and commutes with the arrows in (\cref{diagram: simultaneous_resolution_S}). Any $\sigma\in \Phi_{\RR}$ induces a homeomorphism $\sigma: \chi_{S}^{-1}(\tau)\mapsto \chi_{S}(\sigma^{-1}(\tau))$ and stabilises the exceptional curves. In particular, over $0$, the $\Phi_{\RR}$--action on curves is compatible with the Dynkin symmetry.
\end{itemize}
\label{property: global_symmetry}
\end{property}

Let $i: \bbC^2\to \mathfrak{h}^{\Phi_{\RR}}$ be any \textbf{generic} linear embedding.

\begin{defn}
Let $W$ and $W'$ be the total of spaces in the following pull back families:
\[
\xymatrix{
W \ar[d] \ar[r] & i^*\widetilde{S} \ar[d] & W' \ar[l] \ar[d] \\
\{0\}\times \bbC  \ar[r] 
& \bbC \times \bbC 
& \bbC\times \{1\} \ar[l]
}
\]
\label{def: threefolds}
\end{defn}

\begin{prop}
    $W$ and $W'$ satisfy the properties listed in \Cref{properties: threefolds}.
\end{prop}
\begin{proof}
\begin{itemize}
    \item[(P1-P3)]Since $\mathrm{Im}(i:\bbC^2\to \mathfrak{h}_{\RR})$ lies in the invariant subspace $\mathfrak{h}_{\RR}^{\Phi_{\RR}}$, P1-P3 follows from \Cref{property: global_symmetry};
    
    \item[(P4)] By construction, $Z_\RR=\widetilde{\chi_S}^{-1}(0) \subset W$. To show $Z_\RR$ contains all compact curves of $W$, it suffices to show $\mathfrak{h}_{\RR}^{\Phi_{\RR}}$ is not contained in any hypersurface $H_{\beta}$, in another word, $\langle \beta,-\rangle$ is non-zero on $\mathfrak{h}_{\RR}^{\Phi_{\RR}}$. For this, we use $\langle\beta,v\rangle = \langle\bar{\beta},v\rangle$ if $v\in \mathfrak{h}_{\RR}^{\Phi_{\RR}}$ and notice that the averaging $\bar{\beta}\in \mathfrak{h}_{\RR}^{\Phi_{\RR}}$ is a positive root of $\RR^{\mathrm{ave}}$ and hence is nonzero. The affineness follows from the affineness of $S = e + Ker({\rm ad}f)$;
    
    \item[(P5)] Since $\mathfrak{h}_{\RR}^{\Phi_{\RR}}$ is not contained in any hyperplane $H_{\beta}$, for a generic choice of the embedding
$i:\bbC\times\{1\}\hookrightarrow \mathfrak{h}_{\RR}^{\Phi_{\RR}}$
the curve $i(\bbC\times\{1\})$ meets each $H_{\beta}$ transversely, and avoids all pairwise intersections $H_{\beta}\cap H_{\beta'}$ unless $\beta$ and $\beta'$ lie in the same $\Phi_{\RR}$--orbit.
Consequently, the exceptional curves in $W$ are contained in the fiber of $\widetilde{\chi_S}$ over the (necessarily unique) point
\[
i(\bbC\times\{1\})\cap H_{\beta}.
\]
If $i(\bbC\times\{1\})$ meets a single hyperplane $H_{\beta}$ (and no other $H_{\beta'}$), then the corresponding fiber contains a unique $(-2)$-curve.
If instead $i(\bbC\times\{1\})$ passes through a point lying on several hyperplanes, then by construction this can only happen for hyperplanes
$H_{\beta'}$ with $\beta'\in \Phi_{\RR}\cdot \beta$.
In this case the fiber contains $|\Phi_{\RR}\cdot\beta|$ distinct $(-2)$-curves, naturally indexed by the roots in the orbit:
we denote by $C_{\beta'}$ the component which deforms along $H_{\beta'}$ in the family. Moreover, it can be checked that if $\beta'= g\beta$ for $1\neq g\in\Phi_{\RR}$, then $\beta$ and $\beta'$ are orthogonal in $\RR$, therefore $C_{\beta}$ and $C_{\beta'}$ do not intersect in $W'$. Finally, by transversality of the intersection, each $C_{\beta}$ are $(-1,-1)$-curves.

    \item[(P6)]  $\Phi_{\RR}$ acts trivially on $\cN_{Z_\RR\setminus W}$ by construction. On $W'$, the fixed points lie on $C_{\beta}$ corresponding to a fixed positive root $\beta\in \RR^{\mathrm{fold}}$ by \cref{property: global_symmetry}. In this case let $\tau_\beta = H_{\beta}\cap i(\bbC\times\{1\})$ and 
    \beq
    \tilde{S}_{\beta} = \widetilde\chi_S^{-1}(\tau_\beta)\;,S_{\beta} = \chi_S^{-1}(\tau_\beta),
    \eeq
    then $\pi_S: \tilde{S}_{\beta}\to S_{\beta}$ is the resolution of an $A_1$ singularity; let $C_\beta\subset \tilde{S}_{\beta}$ be the exceptional curve. To check the tangent weights, we use the short exact sequence
    \begin{equation}
        0 \to \cN_{C_{\beta}\setminus \tilde{S}_{\beta}} \to \cN_{C_{\beta}\setminus W'}\to \cN_{\tilde{S}_{\beta}\setminus W'}\to 0
    \end{equation}
    By construction, $\cN_{C_{\beta}\setminus \tilde{S}_{\beta}}$ is trivial with trivial $\Phi_R$ action. It then suffices to show that $S_{\beta}/\Phi_\RR$  is a $\mathrm{A}_3$ singularity $(\bbC^2/\bbZ_2)/(\bbZ_4/\bbZ_2)$ when $\Phi_R\cong \bbZ_2$ and an $A_5$ singularity $(\bbC^2/\bbZ_2)/(\bbZ_6/\bbZ_2)$ when $\Phi_R\cong \bbZ_3$, therefore $\Phi_{\RR}$ acts on $\tilde{S}_{\beta}\cong T^\ast\bbP^1$ is induced by the action $z\mapsto \zeta z$, where $\zeta$ is a primitive root of unity of order $|\Phi_\RR|$. To see this, we notice that, by \cite[Chap3]{MR3920908}, $S_{\beta}/\Phi_\RR$ is a simple singularity and therefore there is a Kleinian subgroup $G<SL(2;\bbC)$ such that $\Phi_{\RR}\cong G/\bbZ_2$. By the classification of Kleinian subgroups, $G\cong \bbZ_4$ and $G\cong \bbZ_6$ are the only Kleinian subgroups of order $4$ and $6$ respectively.
    
    \item[(P7)]  By construction, $W^{\rm aff}$ is an complete intersections in an affine variety $S$, hence is Calabi--Yau by adjunction formula. Moreover, since the resolution $W\to W^{\rm aff}$ is an simultaneous resolution over $\bbC$ and moreover  is a minimal resolution of Kleinian singularities over each $b\in \bbC$ hence crepant, $W\to W^{\rm aff}$ is crepant. From this it follows that $W$ is Calabi--Yau. The same argument applies to $M'$. To check that $\Phi_{\RR}$ preserves Calabi--Yau forms, it suffices to check $det\big(g: \omega_x \to \omega_x\big)=1$ at any fixed point $x$, since $W$ and $W'$ are connected. For $W$, a fixed point over $\{0\}\times\{0\}$ satisfy this property, since $g$ acts trivially on the base of the family and gives an $\mathrm{A}_{|\Phi_{\RR}|-1}$-singularity to the quotient surface. For $W'$, this has been shown in (P6). 
\end{itemize}
\end{proof}
\subsection{Multi-cover formula for $[\Tot(\cO_{\bbP^1}(-1)\oplus \cO_{\bbP^1}(-1))/\bbZ_n]$}
\label{subsec: multi_cover_formula}

When $n=1$, the multicover formula in \cref{prop: multicover_formula} is the famous Aspinwall--Morrison formula \cite{MR1204770}. When $n > 1$, it can be obtained from Johnson--Pandharipande--Tseng's computation in \cite{JPT_localP1_orbifolds} about local orbifold Gromov-Witten invariants on orbifold $\bbP^1$, which we explain now.

\begin{lem}
\label{lem:quotient_identification}
Consider the action of $\bbZ_n$ on $\bbP^1$ and on the total space $\Tot(\cO_{\bbP^1}(-1)\oplus \cO_{\bbP^1}(-1))$ defined in \cref{prop: multicover_formula}. The quotient stack $[\bbP^1/\bbZ_n]$ is isomorphic to the root stack $\bbP^1[n,n]$ of \cite{JPT_localP1_orbifolds}. Moreover, the two line bundles on the quotient stack induced by the summands of $\Tot(\cO_{\bbP^1}(-1)\oplus \cO_{\bbP^1}(-1))$ coincide with the duals $L_0^\vee$ and $L_\infty^\vee$ defined in \cite{JPT_localP1_orbifolds}.
\end{lem}
\begin{proof}
We compare the stacks and line bundles respectively.
\begin{enumerate}
    \item[(1)] \emph{$[\bbP^1/\bbZ_n]\simeq \bbP^1[n,n]$.}
The quotient stack $[\bbP^1/\bbZ_n]$ has coarse space $\bbP^1$ with orbifold structure given as: at $0$ and $\infty$, a neighborhood is uniformised by the branched covering map $z\mapsto z^n$. Therefore, it is isomorphic to the complex orbicurve $\left(\bbP^1,(0,\infty),(n,n)\right)$ defined in \cite[Definition\,2.2.2]{MR1950941}, which by \cite[Example\,2.4.5]{MR2306040}, coincide with the root stack $\bbP^1[n,n]$.
    \item[(2)]\emph{Identification of line bundles.}
$L_0$ and $L_{\infty}$ in \cite{JPT_localP1_orbifolds} are defined as line bundles associated to the divisors:
\beq
L_0 = \cO_{[\bbP^1/\bbZ_n]}([0/\bbZ_n]),\, L_\infty = \cO_{[\bbP^1/\bbZ_n]}([\infty/\bbZ_n]),
\eeq
It amounts to showing that their dual bundle coincide, as a $\bbZ_n$ equivariant bundle, with the summands of $\Tot(\cO_{\bbP^1}(-1)\oplus \cO_{\bbP^1}(-1))$. To prove it, recall that the $\bbZ_n$ action is given by
\beq
\zeta[x:y] \;=\;[\zeta x:y],
\eeq
therefore, the relative weight of $x/y$ is $1$. Since $0$ and $\infty$ are cut out by $x$ and $y$ respectively, $x$ (resp. $y$) needs to have weight $0$ as global sections of $L_0$ (resp. $L_{\infty}$). This coincides with the construction in \cref{prop: multicover_formula}.

\end{enumerate}

\smallskip
\noindent
\end{proof}

Let $\overline{\cK}_{0,0}\!\big(\mathbb P^{1}[n,n],d\big)$ be the moduli stack of genus-$0$ twisted stable maps
$f\colon C\to \mathbb P^{1}[n,n]$ of degree $d/n$ (in the convention of \cref{subsec:setup-notation}),
where the domain $C$ has no stacky points. Viewing local invariants as twisted invariants, we have
\beq
\langle\rangle_{d/n}^{[\Tot(\cO_{\bbP^1}(-1)\oplus \cO_{\bbP^1}(-1))/\bbZ_n]} = \int_{[\overline \cK_{0}(\mathbb P^1[n,n],d)]^{\mathrm{vir}}}
e\!\left(R^1\pi_* f^*(L_0^\vee \oplus L_\infty^\vee)\right)
\eeq

Therefore the invariant is exactly the one studied by Johnson--Pandharipande--Tseng. In fact they give a formula for a generating function of
\[
C_{g,\alpha+\beta}(d)
:=
\int_{[\overline \cK_{g,\alpha+\beta}(\mathbb \bbP^1[a,b],d)]^{\mathrm{vir}}}
e\!\left(R^1\pi_* f^*(L_0^\vee \oplus L_\infty^\vee)\right)
\]
where $g,a,b\in \bbZ_{\geq0}$ and $\alpha=(\alpha_1,\dots,\alpha_r),\beta=(\beta_1,\dots,\beta_s)$ with $1\le \alpha_i\le a-1$ and $1\le \beta_j\le b-1$. It's easy to see that with $a=b=n$, our invariant is their $C_{g,\alpha+\beta}(nd)$ with $\alpha = \beta = \vec{0}$ and can be extracted from their formula:
\begin{thm}[Johnson--Pandharipande--Tseng~\cite{JPT_localP1_orbifolds}]
Let $\alpha=(\alpha_1,\dots,\alpha_r)$ and $\beta=(\beta_1,\dots,\beta_s)$ with
$1\le \alpha_i\le a-1$ and $1\le \beta_j\le b-1$.
\begin{enumerate}
  \item[(1)] $C_{g,\alpha+\beta}(d)=0$ unless
  \[
  d-\sum_{i=1}^r \alpha_i \equiv 0 \pmod a
  \qquad\text{and}\qquad
  d-\sum_{j=1}^s \beta_j \equiv 0 \pmod b .
  \]
  \item[(2)] If the congruences hold, then
  \begin{equation}\label{eq:JPT21}
  \sum_{g\ge 0} C_{g,\alpha+\beta}(d)\,u^{2g}
  =
  (-1)^{\,r-\frac{\sum_{i=1}^r \alpha_i}{a} + \frac{d}{a} \;+\; s-\frac{\sum_{j=1}^s \beta_j}{b} + \frac{d}{b}}\;
  \frac{d^{\,r+s-3}}{a^{\,r-1} b^{\,s-1}}
  \left(\frac{ud/2}{\sin(ud/2)}\right)^{\!2}.
  \end{equation}
\end{enumerate}
\end{thm}
Now we can conclude the proof:
\begin{cor}
\beq
C_{0,\vec{0}}(nd) = (-1)^{2n}\frac{(nd)^3}{n^2} = \frac{1}{nd^3}
\eeq
\label{cor: multi_cover_formula}
\end{cor}

\begin{proof}
Noticing that 
\beq
\left(\frac{ud/2}{\sin(ud/2)}\right)^{\!2} = 1+ O(u^2), 
\eeq
$C_{0,\vec{0}}(nd)$ can be read from $g=0$ term of \eqref{eq:JPT21}.
\end{proof}

\section{A crepant resolution conjecture for $X_{\RR^{\mathrm{fold}}}$}
\label{sec: CRC}

We have so far described the contribution from the \emph{untwisted sector} to
$\mathrm{QH}^*_{\mathbb C^{\times}}(\cX_{\RR^{\mathrm{fold}}})$.
In this section we formulate a conjecture for the full (orbifold) quantum cohomology and present supporting evidence.
Our strategy is to compare it with the quantum cohomology of a crepant resolution of the coarse moduli space
$X_{\RR^{\mathrm{fold}}}$.

\subsection{Crepant resolution of $X_{\RR^{\mathrm{fold}}}$}
\label{subsec: crepant_resolution_of_foldings}
As explained in the introduction, the surface $X_{\RR^{\mathrm{fold}}}$ has $N_{\RR^{\mathrm{fold}}}$ isolated Kleinian singularities: each is of type $\mathrm{A}_1$ when $\Phi_\RR \cong \mathbb Z_2$, and of type $\mathrm{A}_2$ when
$\Phi_\RR \cong \mathbb Z_3$. This means $X_{\RR^{\mathrm{fold}}}$ admits a crepant resolution by taking the minimal resolutions of these singularities:
\begin{equation}
\psi \colon Z_{\RR^{\rm res}} \longrightarrow X_{\RR^{\mathrm{fold}}}.
\end{equation}

The setup of the \emph{Crepant Resolution Conjecture} \cite{MR2234886, MR2483931} therefore applies and predicts that the orbifold quantum cohomology of $\cX_{\RR^{\mathrm{fold}}}$ can be obtained from the quantum cohomology of $Z_{\RR^{\rm res}}$ by analytic continuation together with an explicit affine-linear isomorphism $H^*\bigl(Z_{\RR^{\rm res}}\bigr)\ \cong\ H^*_{\mathrm{CR}}(\cX_{\RR^{\mathrm{fold}}}).$ One may therefore obtain the full $\mathrm{QH}_{\bbC^{\times}}^*(\cX_{\RR^{\mathrm{fold}}})$ from $\mathrm{QH}_{\bbC^{\times}}^*(Z_{\RR^{\rm res}})$ by taking such an affine map.

The quantum cohomology of $Z_{\RR^{\rm res}}$ is completely clear. By \cite{MR584445}, $Z_{\RR^{\rm res}}$ is in fact the minimal resolution of a Kleinian surface singularity $\mathbb C^2/\Gamma_{\RR^{\rm res}}$, where the root system $\RR^{\rm res}$ associated to $\RR^{\mathrm{fold}}$ is recorded in Table~\ref{tab:resolution_of_folded_surfaces}:

\begin{table}[H]
\centering
\begin{tabular}{|c|c|}
\hline
$\RR^{\mathrm{fold}}$ & $\RR^{\mathrm{res}}$ \\
\hline
$\mathrm{B}_{n}$ & $\mathrm{D}_{n+1}$ \\
$\mathrm{C}_{m}$  & $\mathrm{D}_{2m}$ \\
$\mathrm{E}_6$      & $\mathrm{E}_7$ \\
$\mathrm{D}_4$      & $\mathrm{E}_6$ \\
\hline
\end{tabular}
\caption{minimal resolution types for the folded surfaces.}
\label{tab:resolution_of_folded_surfaces}
\end{table}

$\mathrm{QH}^*_{\mathbb C^{\times}}(Z_{\RR^{\rm res}})$ therefore has been studied in detail 
by Bryan-Gholampour \cite{MR2411404}. In this section, we formulate a precise conjecture, in our examples, for the affine-linear transformation relating
$\mathrm{QH}^*_{\mathbb C^{\times}}(Z_{\RR^{\rm res}})$ and $\mathrm{QH}^*_{\mathbb C^{\times}}(\cX_{\RR^{\mathrm{fold}}})$.
\subsection{Crepant resolution conjecture for $X_{\RR^{\mathrm{fold}}}$}
\label{subsec: CRC_foldings}
The affine-linear transformation will be written as blocked affine map under natural splitting of cohomologies.

Let $1\leq k \leq N_{\RR^{\mathrm{fold}}}$. Denote by $p_k\in Z_\RR$ the fixed points, and $s_k\in X_{\RR^{\mathrm{fold}}}$ the images of $p_k$ under the quotient map. Write $E^{(k)}\subset Z_{\RR^{\rm res}}$ for the exceptional configuration over $s_k$. Decomposing into irreducible components, $E^{(k)} = E^{(k)}_1$ when $\Phi_\RR \cong \bbZ_2$ and $E^{(k)} = E^{(k)}_1 \sqcup E^{(k)}_2$ when $\Phi_\RR \cong \bbZ_3$.

As in \eqref{eq: CRcohomology}, the definition of Chen--Ruan cohomology yields a natural splitting of graded vector spaces
\beq 
H^*_{\mathrm{CR},\bbC^{\times}}(\cX_{\RR^{\mathrm{fold}}})=H_{\bbC^{\times}}^*(\cX_{\RR^{\mathrm{fold}}})\bigoplus_{k=1}^{N_{\RR^{\mathrm{fold}}}}  V_{s_k},
\label{eq: splitting_OP}
\eeq
where $V_{s_i}$ is spanned by fundamental classes of components of $I\cX_{\RR^{\mathrm{fold}}}$ that are supported on $s_i$.

On the other hand, a corresponding splitting respecting $\psi$ can be made as 
\beq
H^*_{\bbC^{\times}}\bigl( Z_{\RR^{\rm res}}\bigr)=\psi^*\left(H^*_{\bbC^{\times}}(X_{\RR^{\mathrm{fold}}})\right)\bigoplus_{k=1}^{N_{\RR^{\mathrm{fold}}}} W_{s_k},
\label{eq: splitting_LR}
\eeq

where $W_{s_i}$ is spanned by exceptional classes supported on $E_i$.

\begin{defn}[Affine identification $\Phi$]
\label{def:affine trans}
Define an affine map
\[
\Phi:\ H^*\bigl( Z_{\RR^{\rm res}}\bigr)\ \longrightarrow\ H^*_{\mathrm{CR}}(\cX_{\RR^{\mathrm{fold}}})
\]
as follows. On $\psi^*H^*(Z_{R'})$ we set $\Phi=\mathrm{id}$ via the isomorphism 
\beq
H^*(\cX_{\RR^{\mathrm{fold}}}) \cong H^*(X_{\RR^{\mathrm{fold}}})
\eeq

For each $k$, we define $\Phi|_{W_{s_k}}$ to be:
\begin{enumerate}
    \item[(1)] When $(\RR,\RR^{\mathrm{fold}}) = (A_{2n-1}, B_n), (D_{m+1},C_m)$ or $(E_6,F_4)$,
    \beq
    t^{(k)}_1: = -\pi\ri\, + \ri\, x^{(k)}_1\, ;
    \eeq
    \item[(2)] When $(\RR,\RR^{\mathrm{fold}}) = (D_4,G_2)$, 
    \beq
    t^{(k)}_i: = \frac{2 \pi \ri\,}{3} + \frac{2\,\ri\,}{3}\sum_{j=1}^2  \sin(\frac{j}{3}\pi)\zeta^{-ij}x^{(k)}_j\, , \, i=1,2
    \eeq
    where $\zeta=\re^{2\pi\ri/3}$.
\end{enumerate}

\end{defn}

\begin{conj}[Crepant resolution prediction for $\cX_{\RR^{\mathrm{fold}}}$]
\label{conj:crc-YRp}
After analytic continuation and affine change of variable as in \cref{def:affine trans}, there is an isomorphism 
\beq
\mathrm{QH}_{\bbC^\times}^*\bigl( Z_{\RR^{\rm res}}\bigr)\ \cong\ \mathrm{QH}^*_{\mathrm{orb},\bbC^\times}\bigl(\cX_{\RR^{\mathrm{fold}}}\bigr).
\eeq
\end{conj}

\subsection{Chen--Ruan cohomology ring of $\cX_{\RR^{\mathrm{fold}}}$}
\label{subsec: quantum_corrected_ring}

Let $\Delta^{+}_{\psi} = \mathrm{Ker}(\psi_*)\cap {\rm Eff}(Z_{\RR^{\rm res}})$ be the semigroup of effective curve classes that is contracted by $\psi$. Following Ruan and Perroni \cite{MR2234886, MR2360646}, we define : 
\begin{defn}[Quantum corrected ring of $Z_{\RR^{\rm res}}$]\label{def: corrected_cup} 
Let $Q_i^{(k)}$ be the quantum parameter with respect to $E_i^{(k)}$. For any $\phi_1,\phi_2,\phi_3 \in H^*(Z_{\RR^{\rm res}})$, define the quantum corrected $3-$point function $\langle\phi_1,\phi_2,\phi_3\rangle_{\rm qc}(q)$ as
   
\beq
\langle\phi_1,\phi_2,\phi_3\rangle_{\rm qc}(q) =  \sum_{\beta\in\Delta^{+}_{\psi}}Q^{\beta}\langle\phi_1,\phi_2,\phi_3\rangle^{Z_{\RR^{\rm res}},\bbC^{\times}}_{0,3,\beta},
\eeq

and quantum corrected product $\cup_{\psi}$ by
\beq
(\phi_1\cup_{\psi}\phi_2,\phi_3)^{Z_{\RR^{\rm res}}}= (\phi_1\cup\phi_2,\phi_3)^{Z_{\RR^{\rm res}}} + \langle\phi_1,\phi_2,\phi_3\rangle_{\rm qc}(q).
\eeq
\end{defn}

\begin{prop}
\label{prop: quantum_corrected_product}
    \beq
    (\phi_1\cup_{\psi}\phi_2,\phi_3)^{Z_{\RR^{\rm res}}} = \nu\, \sum_{\beta\in \RR^{\phi,+}} \langle\phi_1,\beta\rangle \langle\phi_2,\beta\rangle\langle\phi_3,\beta \rangle \frac{1 + Q^{\beta}}{1 - Q^{\beta}},
    \eeq
where
    \beq
    \RR^{\phi,+}= \mathrm{Ker}(\psi_*) \cap \RR^{\rm res,+} \cong 
    \begin{cases}
        \prod_{i=1}^{N_{\RR^{\mathrm{fold}}}} \RR^{+}_{A_1} & \text{if } \Phi_\RR \cong \bbZ_2,\\
        \prod_{i=1}^{N_{\RR^{\mathrm{fold}}}} \RR^{+}_{A_2} & \text{if } \Phi_\RR \cong \bbZ_3.
    \end{cases}
    \label{eq: contracted_roots}
    \eeq
under the identification between $H_2(Z_{\RR^{\rm res}}; \bbZ) $ and the root lattice of $\RR^{\rm res}$.
\end{prop}
\begin{proof}
    Notice that the quantum-corrected ring is just a specialisation of the full quantum cohomology ring to the contribution from contracted curve classes, so the result follows from comparison with \cite{MR2411404}. By \cite{MR2411404}, only positive roots contribute to nontrivial Gromov--Witten invariants. It then remains to describe the contracted roots as in \eqref{eq: contracted_roots}, which can be checked directly.
\end{proof}    
For projective surfaces with a single $A_1$ or $A_2$ singularity, or in general projective varieties with transversal $A_1$ or $A_2$ singularities, Perroni \cite{MR2360646} proves that under an explicit change of basis and specialisation of the quantum parameters, there is a ring isomorphism between their orbifold cohomology and certain quantum-corrected cohomology of their crepant resolutions, where the same splitting as \eqref{eq: splitting_OP} and \eqref{eq: splitting_LR} is used.

In our setup, $X_{\RR^{\mathrm{fold}}}$ is non-compact and admits more than one $A_1$ or $A_2$ singularity. Moreover, an equivariant variant is used to define (quantum) cohomology. Nevertheless, the same construction yields a ring isomorphism.

\begin{prop}
\label{prop:perroni-evidence-YRp}
Equip $H^*( Z_{\RR^{\rm res}})$ with the quantum-corrected product $\cup_\psi$ defined in \cref{def: corrected_cup}.
Then the map $H^*(Z_{\RR^{\rm res}})(q)\to H^*_{\mathrm{CR}}(\cX_{\RR^{\mathrm{fold}}})$ defined by the following formulas:
\begin{enumerate}
    \item[(1)] when $(\RR,\RR^{\mathrm{fold}}) = (A_{2n-1}, B_n), (D_{m+1},C_m)$ or $(E_6,F_4)$,
    \beq
    E^{(k)}_1\mapsto 2 \ri\, \, e^{(k)}_1\, , \, Q^{(k)}= -1;
    \eeq
    \item[(2)] when $(\RR,\RR^{\mathrm{fold}}) = (D_4,G_2)$, 
    \beq
    E^{(k)}_i \mapsto 2\ri\,\sum_{j=1}^2  \sin\Big(\frac{j}{3}\pi\Big)\zeta^{ij} \, e^{(k)}_j\, , \, Q^{(k)}_i= \zeta
    \eeq
    where $\zeta=\re^{2\pi\ri/3}$,
\end{enumerate}
is a ring isomorphism.
\end{prop}

\begin{proof}

By construction, the linear map in the statement agrees, on each block \[\psi^*\left(H^*_{\bbC^{\times}}(X_{\RR^{\mathrm{fold}}})\right)\oplus W_{s_k},\]  with the map appearing in Perroni’s formula \cite[Proposition 6.2]{MR2360646}. Using the same proof as in \cite{MR2360646} with the equivariant variant, a ring isomorphism therefore holds on each block:
\beq
\psi^*\left(H^*_{\bbC^{\times}}(X_{\RR^{\mathrm{fold}}})\right)\oplus W_{s_k} \cong H_{\bbC^{\times}}^*(\cX_{\RR^{\mathrm{fold}}})\oplus  V_{s_k}, \forall k.
\eeq
 Consequently, to reduce to Perroni’s result, it suffices to verify the following two claims:

\textbf{Claim 1}. For $1\leq k\leq N_{\RR^{\mathrm{fold}}}$, the space $\psi^*\left(H^*_{\bbC^{\times}}(X_{\RR^{\mathrm{fold}}})\right)\oplus W_{s_k}$ (resp. $H_{\bbC^{\times}}^*(\cX_{\RR^{\mathrm{fold}}})\oplus  V_{s_k}$) is closed under $\cup_{\psi}$ (resp. $\star_{\mathrm{CR},\bbC^\times}$).

\textbf{Claim 2}. $W_{s_{k_1}}\cup_{\psi} W_{s_{k_2}}=0$ and $V_{s_{k_1}}\star_{\mathrm{CR}}V_{s_{k_2}}=0$ if $k_1\neq k_2$.

The resolution side follows from \Cref{prop: qproduct_degree0}: from it we see that $(\phi_1\cup_{\psi}\phi_2,\phi_3)^{Z_{\RR^{\rm res}}}$ is zero as long as two of $\phi_1,\phi_2,\phi_3$ are supported over fibres of distinct singular points.

On the orbifold side, we first show Claim 2. For classes $\delta_1$ and $\delta_2$ supported over distinct singular points, by construction,
\beq
\delta_1\star_{\rm CR}\delta_2
  \;\coloneqq\;
  \mu_*\Bigl( e_1^*\delta_1 \;\cup\; e_2^*\delta_2 \;\cup\; e_{g,h}\Bigr),
\eeq
Since $e_1^*(\delta_1)$ and $e_2^*(\delta_2)$ have different support, $e_1^*\delta_1 \;\cup\; e_2^*\delta_2 = 0$.  Hence one has $V_{s_{k_1}}\star_{\mathrm{CR}}V_{s_{k_2}}=0$ if $k_1\neq k_2$.

It remains to show Claim 1 on the orbifold side. First, the untwisted part is by itself closed, so is each component $V_{s_k}$. In order to show Claim 1, it suffices to show, for any $\phi\in H_{\bbC^{\times}}^*(\cX_{\RR^{\mathrm{fold}}})$, and $\delta_1$ and $\delta_2$ supported over distinct singular points,
\beq
(\phi\star_{\mathrm{CR},\bbC^\times} \delta_1,\delta_2)^{\cX_{\RR^{\mathrm{fold}}}} = 0
\eeq

By associativity
\beq
(\phi \star_{\mathrm{CR},\bbC^\times} \delta_1,\delta_2)^{\cX_{\RR^{\mathrm{fold}}}} = (\phi,\delta_1\star_{\mathrm{CR},\bbC^\times}\delta_2)^{\cX_{\RR^{\mathrm{fold}}}},
\eeq
it reduces to Claim 2. 
\end{proof}

\begin{rmk}
This result can be viewed as evidence for \cref{conj:crc-YRp}. Indeed, assuming that \cref{conj:crc-YRp} holds, this result can be obtained by taking a certain limit, as follows. On the resolution side, the quantum modified ring can be viewed as the Frobenius algebra at the \textbf{partial large radius limit}:
\beq
\mathrm{Re}\langle\tau,E\rangle\to -\infty,
\eeq
in the big quantum product, or 
\beq
Q_E\to 0
\eeq
in the small quantum product, for all uncontracted curves $E\subset Z_{\RR^{\rm res}}$. 

Correspondingly, on the orbifold side, all curves in $X_{\RR^{\mathrm{fold}}}$ are of the form $\psi(E)$, with $E$ an uncontracted curve in $Z_{\RR^{\rm res}}$, under the \textbf{large radius limit}
\beq
\mathrm{Re}\langle\tau,E\rangle\to -\infty,
\eeq
in the big quantum product, or 
\beq
Q_{\psi(E)}\to 0,
\eeq
in the small quantum product, $\mathrm{QH}^*_{\mathrm{orb}}\bigl(\cX_{\RR^{\mathrm{fold}}}\bigr)$ degenerates to the ordinary Chen--Ruan cohomology $H^*_{\mathrm{CR}}(\cX_{\RR^{\mathrm{fold}}})$. 

To finish the remark, one can check that the linear map in \cref{prop:perroni-evidence-YRp} is the inverse of the linear part of \cref{def:affine trans}.
\end{rmk}
\subsection{Compatibility with a Fourier Mukai transform}
\label{subsec: FM_transform}
Ruan's conjecture is closely related to the (largely open) DK conjecture~\cite{MR1957019,MR1949787}, which predicts that $K$-equivalent varieties/orbifolds should be derived equivalent and hence have isomorphic $K$-groups. A conceptual link is provided by Iritani's integral structure and central charge on $K$-theory~\cite{MR2683208}: he explains how a derived equivalence induces a concrete affine transformation on cohomology, compatible with the relevant enumerative structures.

We first fix the integral-structure conventions. For a smooth Deligne--Mumford stack $\cZ$ of complex dimension $d$ and $V\in K_c(\cZ)$, we use Iritani's cohomology-valued class \cite[Equation~(12)]{MR2683208}
\beq
\Psi_{\cZ}(V)
\coloneqq(2\pi)^{-d/2}\widehat{\Gamma}(T\cZ)
\cup(2\pi\ri)^{\deg/2}\operatorname{inv}^*
\widetilde{\operatorname{ch}}(V),
\qquad d=\dim_{\bbC}\cZ,
\label{eq: Iritani_Psi}
\eeq
where $\deg$ is the ordinary cohomological degree on the inertia stack. Following \cite[Proposition~3.28]{MR2683208}, the integral central charge is then
\beq
\cF_{\cZ}(V)(\tau)
=(2\pi)^{-d/2}\ri^{-d}
\left(\left.J_{\cZ}(\tau,z)\right|_{z=-1},
\Psi_{\cZ}(V)\right)_{\mathrm{orb}}.
\label{eq: integral_central_charge}
\eeq
For the noncompact spaces considered here, we take $V\in K_c(\cZ)$, so that $\Psi_{\cZ}(V)$ is compactly supported and the orbifold pairing in \eqref{eq: integral_central_charge} is well-defined.

\begin{rmk}
The $J$-function $J_{\cZ}(\tau,z)$ used by Iritani in this setup is normalised with expansion (\cite[Equation~(55)]{MR2683208})
\beq
J_{\cZ}(\tau,z)
=1+\frac{\tau}{z}+O(z^{-2}).
\label{eq: normalized_J_ch4}
\eeq
\end{rmk}

\begin{conj}[Following Iritani {\cite[Corollary~3.30]{MR2683208}} {\cite[Proposal~5.7]{MR2553377}}]
Let $\cX$ be an orbifold with coarse moduli space $X$, and let $\phi\colon Z\to X$ be a crepant resolution. There exists an isomorphism
\beq
\mathbb{U}_{K}\colon K_c(Z)\longrightarrow K_c(\cX)
\eeq
induced by a Fourier--Mukai equivalence. After analytic continuation, there exists an affine change of flat coordinates $\mathbb{U}_{\rm coh}$ with the following properties:
\begin{enumerate}
    \item[(1)] the quantum $\DD$-modules of $Z$ and $\cX$ are isomorphic in a way that preserves their $\widehat{\Gamma}$-integral structures and is induced by $\mathbb{U}_{K}$;
    \item[(2)] $\mathbb{U}_{\rm coh}$ identifies the corresponding quantum products;
    \item[(3)] for every $V\in K_c(Z)$,
    \beq
    \cF_Z(V)(\tau)
    =\pm\,\cF_{\cX}\bigl(\mathbb{U}_{K}(V)\bigr)
    \bigl(\mathbb{U}_{\rm coh}(\tau)\bigr),
    \eeq
    where the sign $\pm$ is independent of $V$.
\end{enumerate}
\end{conj}

For Kleinian singularities, the derived McKay correspondence gives an isomorphism
\beq
K_c\bigl(\widetilde{\bbC^2/G}\bigr)\ \simeq\ K_0^G(\bbC^2),
\label{eq: derived_mckay}
\eeq
for any finite subgroup $G<\mathrm{SL}(2;\bbC)$. Iritani explains how the associated integral central charges recover the coordinate changes predicted by the quantum McKay correspondence; see \cite[\S 3.9]{MR2683208}.

In the remainder of this section, we verify this predicted central-charge compatibility for the Fourier--Mukai transform \eqref{eq: FMtrans} and the affine change of variables in \cref{def:affine trans}.

By \cite[Theorem 1.6]{MR3049308}, $\Phi_\RR\mathrm{-Hilb}(Z_\RR)$ is a crepant resolution of $\bbC^2/\Gamma_{\RR^{res}}$ and therefore isomorphic to the minimal resolution $Z_{\RR^{\rm res}}$. Let $\mathcal{U}_{\RR}\subset Z_{\RR^{\rm res}}\times Z_\RR$ be the universal scheme, with projections
\beq
p\colon \mathcal{U}_{\RR}\longrightarrow Z_{\RR^{\rm res}},
\qquad
q\colon \mathcal{U}_{\RR}\longrightarrow Z_\RR.
\eeq
The derived McKay correspondence of Bridgeland--King--Reid \cite{BKR} then implies that the Fourier--Mukai transform

\beq
\label{eq: FMtrans}
 Rq_*\circ p^*: K(Z_{\RR^{\rm res}})\to K^{\Phi_\RR}(Z_\RR),
\eeq

is a derived equivalence. In particular, coherent sheaves on $Z_{\RR^{\rm res}}$ naturally correspond to $\Phi_{\RR}$-equivariant sheaves on $Z_\RR$.

Given a vertex $\rho$ in the Dynkin diagram of $\RR^{res}$, let $C_\rho\subset Z_{\RR^{\rm res}}$ be the corresponding Dynkin curve.

\begin{prop} 
\label{prop: FM_folding_local}
    The image of $\cO_{C_{\rho}}(-1)$ under $Rq_*\circ p^*$, as $\Phi_{\RR}$-equivariant sheaves $\LL_{\rho}$ on $K(Z_\RR)$, are as follows:
    \begin{enumerate}
    \item[(1)] If $C_\rho$ is contracted to a point $x_\rho$ on the quotient $X_{\RR^{\mathrm{fold}}}$ corresponding to a $\Phi_{\RR}$--fixed point $y_\rho\in Z_\RR$, then define $\LL_{\rho}:= k(y_\rho)\otimes \omega$, where $\omega: \Phi_\RR\to \bbC^{\times}$ is the irreducible representation of $\Phi_\RR$ in the local McKay correspondence of $\psi$ at $x_\rho$;
    \item[(2A)] If $\pi^{-1}(\psi(C_{\rho}))$ is a \textbf{free} $\Phi_\RR$-orbit of irreducible curves $\Phi_\RR\cdot E_{\rho}$, where $E_{\rho}\subset Z_\RR$ is any representative. In this case define $\LL_{\rho}:= \sum\limits_{g\in \Phi_\RR}\cO_{g(E_{\rho})}(-1)$;
    \item[(2B)] Otherwise, $\pi^{-1}(\psi(C_{\rho}))$ is a single curve $E_{\rho}$, which is stable under the $\Phi_\RR$ action. In this case define $\LL_{\rho} = \cO_{E_{\rho}}(-|\Phi_{\RR}|)$.
\end{enumerate}
\end{prop}
\begin{proof}
We use the property that the Fourier--Mukai transform of a sheaf can be restricted to a neighbourhood of its support and show $(\text{FM})$ in the two cases respectively:

\begin{enumerate}
    \item[(1)] If $C_{\rho}$ is contracted to a point $x_{\rho}$, then over a $\Phi_\RR$-invariant neighbourhood of $x_{\rho}$, $(p,q)$ can be identified with
\beq
\xymatrix{
&\mathcal{U}\subset \Phi_\RR\mathrm{-Hilb}(\bbC^2)\times \bbC^2\ar[dl]_{p}\ar[dr]^{q}&\\
\widetilde{\bbC^2/\Phi_\RR}   && \bbC^2
}
\label{eq: derived_Mckay}
\eeq
    which is the same as in the local McKay correspondence;

    \item[(2A)] If $\pi^{-1}(\psi(C_{\rho}))$ is a free $\Phi_\RR$-orbit, then $q$ is an isomorphism and $p$ is a covering map over a neighbourhood of $\psi(C_{\rho})$. It follows that $\LL_{\rho}:= \sum\limits_{g\in \Phi_\RR}\cO_{g(E_{\rho})}(-1)$.

    \item[(2B)] If $\pi^{-1}(\psi(C_{\rho}))$ is a single curve $E_{\rho}$, then over a neighbourhood of $\psi(C_{\rho})$, the pair $(p,q)$ can be identified with
\beq
\xymatrix{
&\mathcal{U}\subset Y \times \Tot(\cO(-2))\ar[dl]_{p}\ar[dr]^{q}&\\
Y   && \Tot(\cO(-2))
}
\label{eq: BKR_A1_quotient_Z2},
\eeq  
where $Y = \Phi_\RR\mathrm{-Hilb}(\Tot(\cO(-2)))$ is the minimal resolution of an $A_3$ (resp. $A_5$) singularity when $\Phi_{\RR}\cong \bbZ_2$ (resp. $\bbZ_3$), and $C_{\rho}$ is the component corresponding to the middle vertex in the corresponding Dynkin diagram. Therefore, $p$ is a degree $|\Phi_{\RR}|$ map over $C_{\rho}$ and $q$ is an isomorphism restricting to $p^{-1}(C_{\rho})\cong \bbP^1$. As a consequence, we get
\beq
q_*p^*(\cO_{C_{\rho}}(-1)) = \cO_{E_{\rho}}(-|\Phi_{\RR}|).
\eeq
\end{enumerate}

\end{proof}

To show compatibility with integral structure, we start with a local computation.
\begin{lem}
\label{lem: local_central_charge_2b}
Let $n\geq 2$ and $\zeta=\re^{2\pi\ri/n}$. Consider the cotangent lift of the $\bbZ_n=\langle\zeta\rangle$ action
\[
\zeta^k\circ[z_0,z_1]=[\zeta^kz_0,z_1]
\]
on $\bbP^1$ to $\Tot(\cO_{\bbP^1}(-2))=T^*\bbP^1$. Let
\[
i:[\bbP^1/\bbZ_n]\longrightarrow[\Tot(\cO_{\bbP^1}(-2))/\bbZ_n]
\]
be the zero-section inclusion, and equip $\cO_{\bbP^1}(-n)$ with the trivial action on the fibers over fixed points $p_0$ and $p_{\infty}$.

Let $\cF_\tau(V)$ denote the integral central charge \eqref{eq: integral_central_charge}, where $\Psi$ is the cohomology-valued class defined in \eqref{eq: Iritani_Psi}. Then
\beq
\cF_\tau\bigl(i_*\cO_{\bbP^1}(-n)\bigr)
=-\frac{n-1}{n}+\frac{\ri\tau_0}{2\pi n}
+\cF_{p_0,\mathrm{tw}}+\cF_{p_\infty,\mathrm{tw}},
\eeq
with
\beq
\cF_{p_0,\mathrm{tw}}
=\sum_{k=1}^{n-1}\frac{1-\zeta^k}{4\pi n\sin(\pi k/n)}\tau_{0,k},
\qquad
\cF_{p_\infty,\mathrm{tw}}
=\sum_{k=1}^{n-1}\frac{1-\zeta^{-k}}{4\pi n\sin(\pi k/n)}\tau_{\infty,k}.
\eeq
Here the coordinates $\tau_0,\tau_{0,k},\tau_{\infty,k}$ of
\beq
\tau=\tau_0H+\sum_{k=1}^{n-1}\tau_{0,k}e_{0,k}
+\sum_{k=1}^{n-1}\tau_{\infty,k}e_{\infty,k}\in H^*_{\rm CR}([\Tot(\cO_{\bbP^1}(-2))/\bbZ_n])
\eeq
are defined by taking $H=c_1(\cO_{\bbP^1}(1))$, normalized by
\beq
\int_{[\bbP^1/\bbZ_n]}H=\frac{1}{n},
\eeq
while $e_{0,k}$ and $e_{\infty,k}$ are the fundamental classes of the twisted sectors over $p_0$ and $p_\infty$, respectively, labelled by $\zeta^k$.

\end{lem}
\begin{proof}
Write $\cX_n=[\Tot(\cO_{\bbP^1}(-2))/\bbZ_n]$ and let $N=\cO_{\bbP^1}(-2)$ be the normal bundle of the zero section.  Grothendieck--Riemann--Roch for the closed immersion $i$ gives
\beq
\operatorname{ch}(i_*L)\operatorname{Td}(T\cX_n)
=i_*\bigl(\operatorname{ch}(L)\operatorname{Td}(T[\bbP^1/\bbZ_n])\bigr),
\eeq
or, equivalently,
\beq
\operatorname{ch}(i_*L)
=i_*\bigl(\operatorname{ch}(L)\operatorname{Td}(N)^{-1}\bigr).
\eeq
For $L=\cO_{\bbP^1}(-n)$, using $H^2=0$, one has
\beq
\operatorname{ch}(L)=1-nH,
\qquad
\operatorname{Td}(N)=1-H,
\qquad
\operatorname{Td}(N)^{-1}=1+H.
\eeq
Since $i_*(1)=[\bbP^1/\bbZ_n]$ and $i_*(H)=\frac1n[\mathrm{pt}]$, the untwisted component is
\beq
\begin{aligned}
\widetilde{\operatorname{ch}}\bigl(i_*\cO_{\bbP^1}(-n)\bigr)\big|_{(e)}
&=i_*\bigl((1-nH)(1+H)\bigr)\\
&=[\bbP^1/\bbZ_n]-\frac{n-1}{n}[\mathrm{pt}].
\end{aligned}
\eeq

Let $q:\Tot(\cO_{\bbP^1}(-2))\to\bbP^1$ be the projection.  The equivariant Koszul resolution of the zero section is
\beq
0\longrightarrow q^*(L\otimes N^\vee)
\longrightarrow q^*L
\longrightarrow i_*L
\longrightarrow0.
\eeq

By the definition of the orbifold Chern character \cite[2.4--2.5]{MR2683208},
\beq
\widetilde{\operatorname{ch}}\bigl(i_*L\bigr)\big|_{(p,\zeta^k)}
=\operatorname{Tr}(\zeta^k|L_p)
\left(1-\operatorname{Tr}(\zeta^k|N_p^\vee)\right).
\eeq
At $p_0$, the tangent and normal characters are $(\zeta,\zeta^{-1})$, while at $p_\infty$ they are $(\zeta^{-1},\zeta)$.  Together with the trivial fiber action on $L$, this gives
\beq
\widetilde{\operatorname{ch}}\bigl(i_*\cO_{\bbP^1}(-n)\bigr)\big|_{(p_0,\zeta^k)}
=1-\zeta^k,
\qquad
\widetilde{\operatorname{ch}}\bigl(i_*\cO_{\bbP^1}(-n)\bigr)\big|_{(p_\infty,\zeta^k)}
=1-\zeta^{-k}.
\eeq

The untwisted Gamma class restricts to $1$, and on every twisted sector labelled by $\zeta^k$ it is
\beq
\widehat{\Gamma}_k
=\Gamma\left(\frac{k}{n}\right)\Gamma\left(1-\frac{k}{n}\right)
=\frac{\pi}{\sin(\pi k/n)}.
\eeq
Specializing \eqref{eq: Iritani_Psi} to the surface $\cX_n$ gives
\beq
\Psi(V)=\frac1{2\pi}\widehat{\Gamma}(T\cX_n)
\cup(2\pi\ri)^{\deg/2}\operatorname{inv}^*
\widetilde{\operatorname{ch}}(V),
\eeq
where $\deg$ is the ordinary cohomological degree on the inertia stack, its untwisted component is
\beq
\Psi\bigl(i_*\cO_{\bbP^1}(-n)\bigr)\big|_{(e)}
=\ri[\bbP^1/\bbZ_n]+\frac{2\pi(n-1)}{n}[\mathrm{pt}],
\eeq
and its twisted components are
\beq
\Psi\bigl(i_*\cO_{\bbP^1}(-n)\bigr)\big|_{(p_0,\zeta^k)}
=\frac{1-\zeta^{-k}}{2\sin(\pi k/n)}[\mathrm{pt}_{0,k}],
\eeq
\beq
\Psi\bigl(i_*\cO_{\bbP^1}(-n)\bigr)\big|_{(p_\infty,\zeta^k)}
=\frac{1-\zeta^k}{2\sin(\pi k/n)}[\mathrm{pt}_{\infty,k}].
\eeq

To pair these components with the twisted coordinates of $\tau$, by definition of the orbifold Poincar\'e pairing,
\beq
\bigl(e_{j,k},[\mathrm{pt}_{j,n-k}]\bigr)_{\mathrm{orb}}=\frac1n,
\qquad j\in\{0,\infty\}.
\eeq
Thus the coefficient of $\tau_{j,k}$ is read from the $(j,n-k)$-component of $\Psi$.
	By \eqref{eq: normalized_J_ch4}, the part of $\left.J(\tau,z)\right|_{z=-1}$ that pairs nontrivially with the compactly supported components of $\Psi(V)$ is $1-\tau$; the remaining $O(z^{-2})$ terms have no component in the dual ordinary degrees. Since $\dim_{\bbC}\cX_n=2$, the prefactor in \eqref{eq: integral_central_charge} is $-1/(2\pi)$. Hence
	\beq
	\cF_\tau(V)=-\frac1{2\pi}(1-\tau,\Psi(V))_{\mathrm{orb}}
	\eeq
	for $V=i_*\cO_{\bbP^1}(-n)$. Substituting the preceding components gives the constant term $-(n-1)/n$, the untwisted linear term $\ri\tau_0/(2\pi n)$, and the two twisted sums stated above.
\end{proof}

\begin{prop}
\label{prop:FM-crc-compatibility}
The change of variables in \cref{def:affine trans} is compatible with the Fourier--Mukai transform \eqref{eq: FMtrans} under the integral structure central charge in the following sense:
\begin{enumerate}
    \item[(1)] when $(\RR,\RR^{\mathrm{fold}}) = (A_{2n-1}, B_n), (D_{m+1},C_m)$ or $(E_6,F_4)$,
    \beq
    \cF(\cO_{C_{\rho}}(-1)) =  \cF(\LL_{\rho})
    \eeq
    \item[(2)] when $(\RR,\RR^{\mathrm{fold}}) = (D_4,G_2)$, 
    \beq
    \cF(\cO_{C_{\rho}}(-1)) = \pm \cF(\LL_{\rho}),
    \eeq
    where the sign $\pm$ is $-1$ when $C_\rho$ is contracted under $\psi$ and $+1$ in the other cases.
\end{enumerate}
\end{prop}

\begin{proof}
We notice that the image of the central charge of a sheaf depends only on a neighbourhood of its support, where the Fourier-Mukai transform has been described in the proof of \cref{prop: FM_folding_local}. We show the compatibility corresponding to the three cases respectively:

\begin{enumerate}
    \item[(1)] Suppose first that $C_{\rho}$ is contracted to a fixed point. Restricting to a $\Phi_\RR$-invariant neighbourhood, $(p,q)$ is the usual local McKay correspondence for the cyclic quotient singularity of type $A_{|\Phi_\RR|-1}$, hence of type $A_1$ or $A_2$. In this local situation, the change of integral central charges has been computed in \cite{MR2683208}. Comparing with the affine identification in \cref{def:affine trans}, we obtain
    \beq
    \cF(\cO_{C_{\rho}}(-1))=\cF(\LL_{\rho})
    \eeq
    for the order-two foldings in part $(1)$ of the statement. In the $(D_4,G_2)$ case, the affine identification in \cref{def:affine trans} differs from the convention of \cite{MR2683208} by an overall negative sign, so for the contracted curves one obtains
    \beq
    \cF(\cO_{C_{\rho}}(-1))=-\cF(\LL_{\rho}).
    \eeq
    This is the negative sign in part $(2)$ of the statement, and the same sign is used below for the side-block curves in case $(2B)$.

    \item[(2A)] Suppose that $\pi^{-1}(\psi(C_{\rho}))$ is a free $\Phi_\RR$-orbit. From \cite{MR2683208}, we know that
    \beq
    \cF_{Z_{\RR^{\rm res}}}(\cO_{C_{\rho}}(-1)) = -\frac{1}{2 \pi \ri\,}\,t_{\rho}.
    \eeq
    Since the support lies in the free locus of the quotient stack, the orbifold central charge is the ordinary central charge upstairs divided by $|\Phi_\RR|$. Thus
    \beq
    \cF_{\cX_{\RR^{\mathrm{fold}}}}(\LL_{\rho})
    =\frac{1}{|\Phi_\RR|}
    \cF_{Z_{\RR}}\Bigl(\sum\limits_{g\in \Phi_\RR}\cO_{g(E_{\rho})}(-1)\Bigr)
    =-\frac{1}{2 \pi \ri\,|\Phi_\RR|}
    \sum\limits_{g\in \Phi_\RR} t_{g(E_\rho)}.
    \eeq
    On the other hand, by the averaging convention \eqref{def: averaging}, \eqref{eq: second_homology}, and the identity on the untwisted summand in \cref{def:affine trans}, the flat coordinates satisfy
    \beq
    t_{\rho}=
    \frac{1}{|\Phi_\RR|}
    \sum\limits_{g\in \Phi_\RR} t_{g(E_\rho)}.
    \eeq
    Hence $\cF(\cO_{C_{\rho}}(-1))=\cF(\LL_{\rho})$ in case $(2A)$, which is the $+$ sign in the statement.

    \item[(2B)] The integral central charge of $\LL_{\rho}$ is computed in \cref{lem: local_central_charge_2b}. The resolution side is the minimal resolution of an $A_{2|\Phi_\RR|-1}$ singularity, with $C_{\rho}=C_{|\Phi_\RR|}$ the middle curve, and the integral central charge of $\cO_{C_{\rho}}(-1)$ is computed in \cite{MR2683208} as
    \beq
    - \frac{1}{2\pi \ri} \tau\cap C_{\rho}.
    \eeq
    Take $C_1,\dots,C_{2|\Phi_\RR|-1}$ as the $A_{2|\Phi_\RR|-1}$ chain on the resolution. The basis matching the splitting \eqref{eq: splitting_LR} is
    \beq
    \{D=\psi^*(\pi_*(E_\rho)),C_1,\dots,C_{|\Phi_\RR|-1},
    C_{|\Phi_\RR|+1},\dots,C_{2|\Phi_\RR|-1}\},
    \eeq
    where
    \beq
    D=C_{|\Phi_\RR|}
    +\sum_{k=1}^{|\Phi_\RR|-1}\frac{k}{|\Phi_\RR|}C_k
    +\sum_{k=|\Phi_\RR|+1}^{2|\Phi_\RR|-1}
    \frac{2|\Phi_\RR|-k}{|\Phi_\RR|}C_k,
    \eeq
    where the coefficients are determined by the relations
    \beq
    D\cdot C_k=0,\qquad k\neq|\Phi_\RR|.
    \eeq
    Now write
    \beq
    \tau = s\, D + \sum_{k=1}^{|\Phi_\RR|-1} p_k C_k + \sum_{k=|\Phi_\RR|+1}^{2|\Phi_\RR|-1} p_k C_k \in H^*(\widetilde{\Tot(\cO_{\bbP^1}(-2))/\bbZ_{|\Phi_\RR|}}),
    \eeq
    and let
    \beq
    t_k = \tau\cap C_k
    =
    \begin{cases}
    p_{k-1}-2p_k+p_{k+1}&k\neq |\Phi_\RR|,\\
    -\frac{2}{|\Phi_\RR|}\, s + p_{|\Phi_\RR|-1} + p_{|\Phi_\RR|+1}&k=|\Phi_\RR|,
    \end{cases}
    \eeq
    where $p_0=p_{|\Phi_\RR|}=p_{2|\Phi_\RR|}=0$, and $-2/|\Phi_\RR|$ is the intersection number $D\cdot C_{|\Phi_\RR|}$.

	    Notice also that $D\cdot C_{|\Phi_\RR|} = -2 H\cdot \psi(C_{|\Phi_\RR|})$. Since the untwisted degree-two block is one-dimensional and the identification preserves the pairing with the middle curve, this fixes the relation
	    \beq
	    \tau_0 = -2 s
	    \eeq
	    on the untwisted part.

	    It remains to substitute the affine coordinates in \cref{def:affine trans} into the expression for $t_{|\Phi_\RR|}$. Following the notation of \cref{def:affine trans}, we relabel the two local blocks by the fixed points $p_0$ and $p_{\infty}$; the corresponding affine coordinates are $x_j^{(0)}$ and $x_j^{(\infty)}$. When $|\Phi_\RR|=2$, $C_1$ and $C_3$ are $-2$ curves, so
	    \beq
	    p_1=-\frac12 t_1,\qquad p_3=-\frac12 t_3,
	    \eeq
	    and the affine substitution is the one in \cref{def:affine trans}, with $x_1^{(0)}=\tau_{0,1}$ and $x_1^{(\infty)}=\tau_{\infty,1}$. No additional sign is inserted in this case.
	    Thus
	    \beq
	    p_1+p_3
	    =\pi\ri-\frac{\ri}{2}
	    \left(\tau_{0,1}+\tau_{\infty,1}\right).
	    \eeq

	    When $|\Phi_\RR|=3$, the left-hand $A_2$ block satisfies
	    \beq
	    \begin{pmatrix}
	    t_1\\
	    t_2
	    \end{pmatrix}
	    =
	    -
	    \begin{pmatrix}
	    2&-1\\
	    -1&2
	    \end{pmatrix}
	    \begin{pmatrix}
	    p_1\\
	    p_2
	    \end{pmatrix},
	    \qquad
	    p_2=-\frac{t_1+2t_2}{3}.
	    \eeq
	    On this block the sector labels are $x_j^{(0)}=\tau_{0,j}$. 
	    At $p_{\infty}$ the tangent character is inverse to the tangent character at $p_0$. Therefore the local $A_2$ ordering used in \cref{def:affine trans} is opposite to the global ordering of the right-hand side of the chain $C_1-\cdots-C_5$: the ordered global pair is $(t_5,t_4)$, and the sector labels satisfy
	    \beq
	    (x^{(\infty)}_1,x^{(\infty)}_2)=(\tau_{\infty,2},\tau_{\infty,1}).
	    \eeq
	    With this ordering, the right-hand block satisfies
	    \beq
	    \begin{pmatrix}
	    t_5\\
	    t_4
	    \end{pmatrix}
	    =
	    -
	    \begin{pmatrix}
	    2&-1\\
	    -1&2
	    \end{pmatrix}
	    \begin{pmatrix}
	    p_5\\
	    p_4
	    \end{pmatrix},
	    \qquad
	    p_4=-\frac{t_5+2t_4}{3}.
	    \eeq
	    We then substitute the affine coordinates from \cref{def:affine trans}. In the $(D_4,G_2)$ case, the side-block curves belong to the contracted-to-a-fixed-point case treated in $(1)$ above, so the negative sign there gives $(t_1,t_2)=-(t_1^{(0)},t_2^{(0)})$ and $(t_5,t_4)=-(t_1^{(\infty)},t_2^{(\infty)})$.
	    Consequently,
	    \beq
	    \begin{aligned}
	    p_2={}&\frac{2\pi\ri}{3}
	    -\frac{\ri\sqrt{3}}{9}
	    \left((1-\zeta)\tau_{0,1}
	    +(1-\zeta^2)\tau_{0,2}\right),\\
	    p_4={}&\frac{2\pi\ri}{3}
	    -\frac{\ri\sqrt{3}}{9}
	    \left((1-\zeta^2)\tau_{\infty,1}
	    +(1-\zeta)\tau_{\infty,2}\right).
	    \end{aligned}
	    \eeq

	    Substituting the preceding expressions into $-\frac{1}{2\pi\ri}t_{|\Phi_\RR|}$, together with $\tau_0=-2s$, gives exactly the central charge in \Cref{lem: local_central_charge_2b}. Hence $\cF(\cO_{C_{\rho}}(-1)) = \cF(\LL_{\rho})$ in case $(2B)$, which is the $+$ sign in the statement since $C_{\rho}$ is not contracted by $\psi$ in this case.

\end{enumerate}
\end{proof}

\begin{rmk}
The $\rho$-dependent sign in the $(D_4,G_2)$ case, equivalently in the local $|\Phi_{\RR}|=3$ comparison, is not part of Iritani's original proposal. It
appears here because the affine change of variables matches the central charges only after inserting this sign. A conceptual explanation of this sign is still missing, and we leave it for future work.
\end{rmk}
\section{Frobenius structures associated with affine $BCFG$ root systems}
\label{sec: BCFG}

Given any irreducible root system $\Phi$, including the non-simply-laced ones, Dubrovin--Zhang \cite{MR1606165} construct a natural Frobenius structure $\mathcal{M}_{\Phi}$ on the orbit space of the corresponding extended affine Weyl group. Later work of Brini--van Gemst \cite{Brini:2021pix} provides Lie-theoretic Landau--Ginzburg superpotentials for all these Frobenius manifolds. When $\Phi$ is simply-laced, Brini--Ma--Strachan \cite{brini2025dubrovin} relate the quantum cohomology of the minimal resolution of the corresponding Kleinian singularity to the Dubrovin dual of $\MM_{\Phi}$. When $\Phi$ is non-simply-laced, a similar Frobenius algebra has been constructed by Bryan--Gholampour \cite{MR2411404}, but the geometric counterpart is not yet known. In this section, we explain how the untwisted part of the orbifold quantum cohomology $\mathrm{QH}_{\bbC^{\times}}^*(\cX_{\RR^{\mathrm{fold}}})$ provides a geometric counterpart of these Frobenius structures.

Let $\RR^{\mathrm{fold}}$ be one of the non-simply-laced root systems in \cref{tab:folding_root_systems}, and let $\iota: H^*(\cX_{\RR^{\mathrm{fold}}}) \to H^*_{\rm CR}(\cX_{\RR^{\mathrm{fold}}})$ be the restriction.

The first observation is the comparison with Bryan--Gholampour's construction.
\begin{cor}
\label{cor:BG-folding}
After applying 
\beq
    Q^{\bar{\beta}} = \re^{-\langle\overline{\beta},\tau\rangle},
\eeq
$\iota^*\!\big(\mathrm{QH}_{\bbC^\times}^*(\cX_{\RR^{\mathrm{fold}}})\big)$ is isomorphic, as a Frobenius algebra, to the one constructed in \cite[Theorem~6]{MR2411404} for $\RR^{\mathrm{ave}}$.
    
\end{cor}
\begin{proof}
    Notice that \cref{thm: qproduct_quantum} coincides with \cite[Theorem~6]{MR2411404} exactly.
\end{proof}

The second observation extends the simply-laced result of Brini--Ma--Strachan \cite{brini2025dubrovin} to the non-simply-laced cases.
\begin{thm}
\label{thm:BCFG-dubrovin-dual}
    $\iota^*(\mathrm{QH}_{\bbC^{\times}}^*(\cX_{\RR^{\mathrm{fold}}}))$ is isomorphic to the Dubrovin dual $\MM^{\flat}_{\RR^{\mathrm{ave}}}$ of the extended affine Weyl Frobenius manifold $\MM_{\RR^{\mathrm{ave}}}$ corresponding to $\RR^{\mathrm{ave}}$.
\end{thm}
\begin{proof}
By Brini--van Gemst \cite{Brini:2021pix}, $\MM_{\RR}$ is isomorphic to a Landau-Ginzburg(LG) model $(\lambda_{\RR},\phi)$. For a LG model $(\lambda,\phi)$, its Dubrovin dual is another LG model$(\log(\lambda),\phi)$. Having this, the result follows from proving $\lambda_{\RR}|_{\mathfrak{h}_{\RR^{\mathrm{ave}}}} = \lambda_{\RR^{\mathrm{ave}}}$. This fact is briefly mentioned in \cite{Brini:2021pix}, and we explain it here in more details.

We recall the construction of the LG superpotential given any irreducible root system $\Phi$. First, $\Phi$ is associated with a Lie algebra $\mathfrak{g}_{\Phi}$. $\lambda_{\Phi}$ is then constructed by solving zeroes of
\beq
\cP(w_0,\dots,w_{l_{\Phi}};\lambda_{\Phi},\mu) = \cQ^{\rm red}(\chi_i = w_i - \delta_{i,\bar{k}}\frac{\lambda_{\Phi}}{w_0};\mu),
\eeq
where $\cQ^{\rm red}$ is the $(1 - \mu)$-free part of the characteristic polynomial
\beq
 (1 - \mu)^{z_0}\cQ^{\rm red} =\cQ = \mathop{\rm det}\limits_{\rho}(g - \mu 1)\in \bbZ[\chi_1,\cdots,\chi_{l_{\Phi}}][\mu]
\eeq
of a fundamental representation $\rho$ of $\mathfrak{g}_{\Phi}$ with minimal dimension, and $\chi_i(g):= Tr_{\rho_i}(g)$ is the $i^{\rm th}$ fundamental character. Moreover, $\bar{k}$ is a marked node in correspondence (see \cite{brini2025dubrovin}) with the index of the $\bbC^{\times}$ fixed curve in the exceptional curves, shown in \cref{fig:dynkin_symmetries}. We denote the fundamental weight corresponding to the marked node by $\omega^{\rm tw}$. In order to prove $\lambda_{\RR}|_{\mathfrak{h}_{\RR^{\mathrm{ave}}}} = \lambda_{\RR^{\mathrm{ave}}}$, it's enough to show
\begin{enumerate}
    \item $\rho_{\RR}\big|_{\mathfrak{h}_{\RR^{\mathrm{ave}}}}$ contains $\rho_{\RR^{\mathrm{ave}}}$ with multiplicity exactly 1;
    \item $V(\omega^{\rm tw}_{\RR})\big|_{\mathfrak{h}_{\RR^{\mathrm{ave}}}}$ contains $V(\omega^{\rm tw}_{\RR^{\mathrm{ave}}})$ with multiplicity exactly 1.
\end{enumerate}
This is checked directly.

\end{proof}
\def\cprime{$'$} \def\cprime{$'$}

\end{document}